\documentclass[11pt,a4paper]{article}

\usepackage[utf8]{inputenc}
\usepackage[T1]{fontenc}
\usepackage[left=2.5cm,right=2.5cm,top=3cm,bottom=3cm]{geometry}
\usepackage[toc,page]{appendix}
\usepackage{amsmath,amsfonts,amssymb,amsthm, graphicx, dsfont, stmaryrd, subcaption}
\usepackage{enumitem}
\usepackage{hyperref}
\usepackage{float}
\usepackage{multirow}
\usepackage[dvipsnames]{xcolor}
\usepackage{pifont}
\usepackage[normalem]{ulem} 
\usepackage[english]{babel}

\usepackage{mathabx}

\DeclareMathOperator{\e}{e}

\DeclareMathOperator{\Proj}{\mathrm{Proj}}
\DeclareMathOperator{\dist}{\mathrm{dist}}
\DeclareMathOperator{\card}{\mathrm{card}}

\DeclareMathOperator{\A}{\mathcal{A}}

\DeclareMathOperator{\N}{\mathcal{N}}

\DeclareMathOperator{\R}{\mathds{R}}
\DeclareMathOperator{\Integer}{\mathds{N}}
\DeclareMathOperator{\V}{\mathds{V}\mathrm{ar}}
\DeclareMathOperator{\E}{\mathds{E}}
\DeclareMathOperator{\Cov}{\mathds{C}\mathrm{ov}}
\DeclareMathOperator{\Prob}{\mathds{P}}
\DeclareMathOperator{\1}{\mathds{1}}

\DeclareMathOperator\supp{supp}

\DeclareMathOperator{\Distortion}{\mathcal{Q}_{2, N}}

\newcommand{\vertiii}[1]{{\left\vert\kern-0.25ex\left\vert\kern-0.25ex\left\vert #1
\right\vert\kern-0.25ex\right\vert\kern-0.25ex\right\vert}}

\theoremstyle{plain}
\newtheorem{theorem}{Theorem}[section]
\newtheorem{lemme}[theorem]{Lemma}
\newtheorem{proposition}[theorem]{Proposition}
\newtheorem{cor}[theorem]{Corollary}

\theoremstyle{definition}
\newtheorem{definition}[theorem]{Definition}

\newtheorem{remark}[theorem]{Remark}

\numberwithin{equation}{section}

\renewenvironment{abstract}
 {\small
	\begin{center}
		\bfseries \abstractname\vspace{-.5em}\vspace{0pt}
	\end{center}
	\list{}{%
	\setlength{\leftmargin}{20mm}
	\setlength{\rightmargin}{\leftmargin}%
 }%
 \item\relax}
 {\endlist}

\providecommand{\keywords}[1]
{
  {
  \small
  \textbf{\textit{Keywords---}} #1
  }
}

\newcommand*\samethanks[1][\value{footnote}]{\footnotemark[#1]}

\begin{document}

\title{New Weak Error bounds and expansions for Optimal Quantization}
\author{
	\sc Vincent Lemaire
	\thanks{Sorbonne Université, Laboratoire de Probabilités, Statistique et Modélisation, LPSM, Campus Pierre et Marie Curie, case 158, 4 place Jussieu, F-75252 Paris Cedex 5, France.}
	\and
	\sc Thibaut Montes
	\samethanks[1]
	\thanks{The Independent Calculation Agent, The ICA, 5th Floor, 95 Gresham Street, London.}
	\and
	\sc Gilles Pagès
	\samethanks[1]
}
\maketitle

\begin{abstract}
	We propose new weak error bounds and expansion in dimension one for optimal quantization-based cubature formula for different classes of functions, such that piecewise affine functions, Lipschitz convex functions or differentiable function with piecewise-defined locally Lipschitz or $\alpha$-H\"{o}lder derivatives. These new results rest on the local behaviours of optimal quantizers, the $L^r$-$L^s$ distribution mismatch problem and Zador's Theorem. This new expansion supports the definition of a Richardson-Romberg extrapolation yielding a better rate of convergence for the cubature formula. An extension of this expansion is then proposed in higher dimension for the first time. We then propose a novel variance reduction method for Monte Carlo estimators, based on one dimensional optimal quantizers.
\end{abstract}

~ \\

\textit{ 2010 AMS Classification:} 65C05, 60E99, 65C50.


\section*{Introduction}

Optimal quantization was first introduced in \cite{sheppard1897calculation}, Sheppard worked on optimal quantization of the uniform distribution on unit hypercubes. It was then extended to more general distributions with applications to Signal transmission at the Bell Laboratory in the 50's (see \cite{gersho1982special}) and then developed as a numerical method in the early 90's, for expectation approximations (see \cite{pages1998space}) and later for conditional expectation approximations (see \cite{pages2004optimal,bally2001stochastic,bally2003quantization,printems2005quantization}).

In modern terms, vector quantization consists in finding the projection for the $L^p$-Wasserstein distance of a probability measure on $\R^d$ with a finite $p$-th moment on the convex subset of $\Gamma$-supported probability measure, where $\Gamma$ is a finite subset of $\R^d$ and $0 < p < +\infty$. The aim of Optimal Quantization is to determine the set $\Gamma_N := \{ x_1^N, \dots, x_N^N \} \subset \R^d$ with cardinality at most $N$ which minimizes this distance among all such sets $\Gamma$. Formally, if we consider a random vector $X \in L^p(\Prob)$, we search for $\Gamma_N$, the solution to the following problem
\begin{equation*}
	\min_{\Gamma_N \subset \R, \vert \Gamma_N \vert \leq N } \Vert X - \widehat X^{\Gamma_N} \Vert_{_p}
\end{equation*}
where $ \widehat X^{\Gamma_N} $ denotes the projection of $X$ onto $ \Gamma_N $ (often $ \widehat X^{\Gamma_N} $ is denoted by $ \widehat X^N $ in order to alleviate the notations). The term $ \Vert X - \widehat X^{\Gamma_N} \Vert_{_p} $ is often referred to as the distortion of order $ p $. The existence of an optimal quantizer at a given level $N$ has been shown in \cite{graf2000foundations,pages1998space} and in the one-dimensional case if the distribution of $X$ is absolutely continuous with a \textit{log-concave} density then there exists a unique optimal quantizer at level $N$. In the present paper we will consider one dimensional optimal quantizers. Moreover, we are not only interested by the existence of such a quantizer but also in the asymptotic behaviour of the distortion because it is an important feature for the method in order to determine the level of the error introduced by the approximation. The question concerning the sharp rate of convergence of $\Vert X - \widehat X^N \Vert_{_p}$ as $N$ goes to infinity is answered by Zador's Theorem. For $X \in L^{p+\delta} (\Prob)$, $\delta>0$, such that $\Prob_{_{X}} (d \xi) = \varphi(\xi) \cdot \lambda ( d \xi ) + \nu ( d \xi ) $, where $\nu ~ \bot ~ \lambda$ is the singular component of $\Prob_{_{X}}$ with respect to the Lebesgue measure $\lambda$ on $\R^d$, the rate of convergence is given by
\begin{equation*}
	\lim_{N \rightarrow + \infty} N^{\frac{1}{d}} \Vert X - \widehat X^N \Vert_{_p} = \widetilde{J}_{p,d} \bigg[ \int_{\R^d} \varphi^{\frac{d}{d+p}} d \lambda_d \bigg]^{\frac{1}{p} + \frac{1}{d}}
\end{equation*}
where $\varphi$ is the density of $X$, $\lambda_d$ is the Lebesgue measure on $\R^d$ and $\widetilde{J}_{p,d} = \inf_{N \geq 1} N^{\frac{1}{d}} \Vert U - \widehat{U}^N \Vert_{_p}$, $U \overset{\mathcal{L}}{\sim} \mathcal{U} \big( (0,1)^d \big)$. For more insights on the mathematical/probabilistic aspects of Optimal quantization theory, we refer to \cite{graf2000foundations,pages2015introduction}.

The reason for which we are interested in this optimal quantizer is numerical integration. The discrete feature of the optimal quantizer $\widehat X^N$ allows us to define, for every continuous function $f:\R^d \longrightarrow \R$, such that $f(X) \in L^2 ( \Prob )$, the following quantization-based cubature formula
\begin{equation*}
	\E \big[ f(\widehat X^N) \big] = \sum_{i=1}^{N} p_i f(x_i^N)
\end{equation*}
where $p_i = \Prob ( \widehat X^N = x_i^N)$. Indeed, as $\widehat X^N$ is constructed as the best discrete approximation of $X$ in $L^p(\Prob)$, it is reasonable to approximate $\E \big[ f(X) \big]$ by $\E \big[ f(\widehat X^N) \big]$ which is useful for numerical integrations problems.

The problem of numerical integration appears a lot in applied fields, such as Physics, Computer Sciences or Numerical Probability. For example, in Quantitative Finance, many quantities of interest are of the form
\begin{equation*}
	\E \big[ f(S_t) \big] \qquad \textrm{for some } t > 0,
\end{equation*}
where $f:\R^d \longrightarrow \R$ is a Borel function and $(S_s)_{s\in [0,t]}$ is a diffusion process solution to a Stochastic Differential Equation (SDE)
\begin{equation*}
	S_t = S_0 + \int_{0}^{t} b(s,S_s) ds + \int_{0}^{t} \sigma (s,S_s) dW_s, \qquad S_0 = s_0,
\end{equation*}
where $W$ is a standard Brownian motion living on a probability space $(\Omega, \A, \Prob)$ and $b$ and $\sigma$ are Lipschitz continuous in $x$ uniformly with respect to $s \in [0, t]$, which are the standard assumptions in order to ensure existence and uniqueness of a strong solution to the SDE. Since it is often impossible to compute $\E \big[ f(S_t) \big]$ directly, it has been proposed in \cite{pages1998space} to compute an optimal quantizer $\widehat X^N$ of $X$ where $X$ is a random variable having the same distribution as $S_t$ and to use the previously defined quantization-based cubature formula as an approximation.

Another approach, often used in order to approximate $\E \big[ f(X) \big]$, is to perform a Monte Carlo simulation $\widehat I_M := \sum_{m=1}^{M} f(X^m)$, where $(X^m)_{m=1, \dots, M}$ is a sequence of independent copies of $X$. The method's rate of convergence is determined by the strong law of numbers and the central limit theorem, which says that if $X$ is square integrable, then
\begin{equation*}
	\sqrt{M} \Big( \widehat{I}_M - \E \big[ f(X) \big] \Big) \xrightarrow{\mathcal{L}} \N \big(0, \sigma_{f(X)}^2 \big) \quad \mbox{as} \quad M \rightarrow + \infty
\end{equation*}
where $\sigma_{f(X)}^2 = \V \big( f(X) \big)$. One notices that, for a given $M$, the limiting factor of the method is $\sigma_{f(X)}^2$. Hence, a lot of methods have been developed in order to reduce the variance term: antithetic variables, control variates, importance sampling, etc. The reader can refer to \cite{pages2018numerical,glasserman2013monte} for more details concerning the Monte Carlo methodology and the variance reduction methods.

In this paper we propose a novel variance reduction method of Monte Carlo estimator through quantization. Our method innovates in that it uses a linear combination of one dimensional control variates to reduce the variance of a higher dimensional problem. More precisely, we introduce a quantization-based control variates $\Xi^{N}_k$ for $k=1, \dots, d$. If one considers a function $f : \R^d \mapsto \R$, we approximate $\E \big[ f(X) \big]$ by
\begin{equation*}
	\E \big[ f(X) - \langle \lambda, \Xi^{N} \rangle \big]
\end{equation*}
with $\langle \cdot, \cdot \rangle$ the scalar product in $\R^d$ and $(\Xi^{N}_k)_{k=1, \dots, d} := f_k(X_k) - \E \big[ f_k (\widehat X_k^N) \big]$, where $X_k$ is the $k$-th component of $X$, $\widehat X_k^N$ is an optimal quantizer of $X_k$ of size $N$ and $f_k : \R \mapsto \R$ is designed from $f$. Looking closely at the introduced control variates, one notices that we introduce a bias in the approximation. However, as since it is closely linked to weak error, this bias can be controlled. The present paper focuses on the weak error's rate of convergence.

First, we place ourselves in the case where $X$ is a random variable in dimension one and we consider a quadratic optimal quantizer. We work on the rate of convergence of the weak error induced by the expectation approximation by an optimal quantization-based cubature formula for different classes of functions $f$
\begin{equation*}
	\lim_{N \rightarrow + \infty} N^{\alpha} \big\vert \E \big[ f(X) \big] - \E \big[ f ( \widehat X^N ) \big] \big\vert \leq C_{f, X} < + \infty.
\end{equation*}
The first classical result concerns Lipschitz continuous functions. Using directly the Lipschitz continuity property of $f$ and Zador's Theorem a rate of order $\alpha=1$ can be obtained. Moreover, if we consider the supremum among all functions with a Lipschitz constant upper-bounded by $1$, then
\begin{equation*}
	N \sup_{[f]_{_{Lip}} \leq 1} \big\vert \E \big[ f(X) \big] - \E \big[ f ( \widehat X^N ) \big] \big\vert = N \Vert X - \widehat X^N \Vert_{_{1}} \leq N \Vert X - \widehat X^N \Vert_{_{2}} \xrightarrow{N \rightarrow + \infty} C_{f} < + \infty.
\end{equation*}
A faster rate ($\alpha = 2$) can be attained for differentiable functions with Lipschitz continuous derivative, using a Taylor expansion with integral remainder and the following stationarity property of quadratic optimal quantizers
\begin{equation*}
	\E \big[ X \mid \widehat X^N \big] = \widehat X^N.
\end{equation*}
Moreover, considering the supremum among all functions where the Lipschitz constant of the derivative is upper-bounded by $1$, we have
\begin{equation*}
	N^2 \sup_{[f']_{_{Lip}} \leq 1} \big\vert \E \big[ f(X) \big] - \E \big[ f ( \widehat X^N ) \big] \big\vert = \frac{1}{2} N^2 \Vert X - \widehat X^N \Vert^2_{_2} \xrightarrow{N \rightarrow + \infty} C_{f} < + \infty
\end{equation*}
where the limit is given by Zador's Theorem. A detailed summary about this results can be found in \cite{pages2018numerical}.

In the first part of this paper, we extend this improved rate ($\alpha=2$) to classes of less smooths functions in one dimension. These new results enable us to design efficient variance reduction methods in higher dimensional settings with in view applications to option pricing. The new results concerns the following classes of functions
\begin{itemize}
	\item Lipschitz continuous piecewise affine functions with finitely many breaks of affinity. We use the stationarity property of the optimal quantizer on the cells where there is no break of affinity and then we control the error on the remaining cells using results on the local behaviour of the quantizer.

	\item Lipschitz continuous convex functions, using local behaviours results on optimal quantizers. We use a representation formula for convex functions as integrals of Ridge functions combined with the local behaviour result in order to control the error again.

	\item Differentiable functions with piecewise-defined locally Lipschitz derivative. The functions have $K$ breaks of affinity $\{ a_1, \dots, a_K \}$, such that $ - \infty = a_0 < a_1 < \cdots < a_K < a_{K+1} = + \infty $ and the locally Lipschitz property of the derivative is defined by
	      \begin{equation*}
		      \forall k = 0, \dots, K, \quad \forall x,y \in (a_k, a_{k+1}) \quad \vert f' (x) - f' (y) \vert \leq [f']_{_{k, Lip, loc}} \vert x - y \vert \big(g_k(x) + g_k(y)\big)
	      \end{equation*}
	      where $g_k: (a_k,a_{k+1})\to \R_+$ are non-negative Borel functions. We use the locally Lipschitz property of the derivative combined with the $L^r$-$L^s$ distortion Theorem and Zador's Theorem on the cells where there is no break of affinity and then we control the error on the remaining cells using results on the local behaviour of the quantizer.

	\item Differentiable functions with piecewise-defined locally $\alpha$-H\"{o}lder derivative. The functions have $K$ breaks of affinity $\{ a_1, \dots, a_K \}$, such that $ - \infty = a_0 < a_1 < \cdots < a_K < a_{K+1} = + \infty $ and the locally $\alpha$-H\"{o}lder property of the derivative is defined by
	      \begin{equation*}
		      \forall k = 0, \dots, K, \quad \forall x,y \in (a_k, a_{k+1}), \quad \vert f' (x) - f' (y) \vert \leq [f']_{_{k, \alpha, loc}} \vert x - y \vert^{\alpha} \big(g_k(x) + g_k(y)\big)
	      \end{equation*}
	      where $g_k: (a_k,a_{k+1})\to \R_+$ are non-negative Borel functions. For this class of functions, the rate of convergence is of order $1+\alpha$. The result is obtained using the same ideas as in the locally Lipschitz case.
\end{itemize}

Hence, for all this classes of functions, except the last one, we have
\begin{equation*}
	\lim_{N \rightarrow + \infty} N^{2} \big\vert \E \big[ f(X) \big] - \E \big[ f ( \widehat X^N ) \big] \big\vert \leq C_{f, X} < + \infty.
\end{equation*}

In the second part of the paper we deal with the \textit{weak error expansion} of the approximation of $\E \big[ f (X) \big]$ by $\E \big[ f ( \widehat X^N) \big]$. First, we place ourselves in the one dimensional case by considering a twice differentiable function $f : \R \mapsto \R$ with a bounded Lipschitz continuous second derivative and $X : (\Omega, \A, \Prob) \rightarrow \R $. Through a second order Taylor expansion and with the help of Corollary \ref{EB:conditionalStationarity}, Theorem \ref{EB:distortimesf} and the $L^r$-$L^s$ distortion mismatch Theorem we obtain
\begin{equation*}
	\E \big[ f (X) \big] = \E \big[ f ( \widehat X^N ) \big] + \frac{c_2}{N^2} + O (N^{-(2 + \beta)})
\end{equation*}
where $\beta \in (0,1)$. This expression suggests to use a Richardson-Romberg extrapolation in order to \textit{kill} the first term of the expansion which yields
\begin{equation*}
	\E \big[ f ( X ) \big] = \E \Bigg[ \frac{M^2 f (\widehat X^M) - N^2 f (\widehat X^N)}{M^2 - N^2} \Bigg] + O ( N^{-(2+\beta)} ).
\end{equation*}
Second, we present a result in higher dimension when considering a twice differentiable function $f : \R^d \mapsto \R$ with a bounded Lipschitz continuous Hessian, $X : (\Omega, \A, \Prob) \rightarrow \R^d$ with independent components $(X_k)_{k = 1, \dots, d}$ and $\widehat X^{N}$ a product quantizer of $X$ with $d$ components $(\widehat X_k^{N_k})_{k = 1, \dots, d}$ such that $N_1 \times \cdots \times N_d \simeq N$. Using product quantizer allows us to rely on the one dimensional results for quadratic optimal quantizers and in that case we have
\begin{equation*}
	\E \big[ f(X) \big] = \E \big[ f(\widehat X^{N}) \big] + \sum_{k=1}^{d} \frac{c_k}{N_k^2} + O \bigg( \Big( \min_{k=1:d} N_k \Big)^{-(2+\beta)} \bigg).
\end{equation*}

The paper is organized as follows. First we recall some basic facts and deeper results about optimal quantization in Section \ref{EB:section:aboutoptimalquantization}. In Section \ref{EB:section:weakerror}, we present our new results on weak error for some classes of functions. Then, we see in Section \ref{EB:section:RR} how to derive \textit{weak error expansion} allowing us to specify the right hypothesis under which we can use a Richardson-Romberg extrapolation. Finally, we conclude with some applications. The first one is the introduction of our novel variance reduction involving optimal quantizers. The last one illustrates numerically the results shown in Section \ref{EB:section:weakerror} and \ref{EB:section:RR}, by considering a Black-Scholes model and pricing different types of European Options. We also propose a numerical example for the variance reduction.

\section{About optimal quantization ($d=1$)} \label{EB:section:aboutoptimalquantization}

Let $X$ be a $\R$-valued random variable with distribution $\Prob_{_{X}}$ defined on a probability space $ ( \Omega, \A, \Prob )$ such that $X \in L^2 ( \Prob )$.

\begin{definition}
	Let $\Gamma_N = \{ x_1^N, \dots, x_N^N \} \subset \R$ be a subset of size $N$, called $N$-quantizer. A Borel partition $\big( C_i (\Gamma_N) \big)_{i = 1, \dots, N}$ of $\R$ is a Vorono\"{i} partition of $\R$ induced by the $N$-quantizer $\Gamma_N$ if, for every $i = 1, \dots, N$,
	\begin{equation*}
		C_i (\Gamma_N) \subset \big\{ \xi \in \R, \vert \xi - x_i^N \vert \leq \min_{j \neq i }\vert \xi - x_j^N \vert \big\}.
	\end{equation*}
	The Borel sets $C_i (\Gamma_N)$ are called Vorono\"{i} cells of the partition induced by $\Gamma_N$.
\end{definition}

One can always consider that the quantizers are ordered: $x_1^N < x_2^N < \cdots < x_{N-1}^N < x_{N}^N $ and in that case the Vorono\"{i} cells are given by
\begin{equation*}
	C_{k} ( \Gamma_N ) = ( x_{k - 1/2}^N, x_{k + 1/2}^N ], \qquad k = 1, \dots, N-1, \qquad C_{N} ( \Gamma_N ) = ( x_{N - 1/2}^N, x_{N + 1/2}^N )
\end{equation*}
where $\forall k = 2, \dots, N, \, x_{k-1/2}^N := \frac{x_{k-1}^N + x_k^N}{2}$ and $x_{1/2}^N := \inf \big(\supp (\Prob_{_{X}}) \big)$ and $x_{N+1/2}^N := \sup \big(\supp (\Prob_{_{X}}) \big)$.

\begin{definition}
	Let $\Gamma_N = \{ x_1^N, \dots, x_N^N \}$ be an $N$-quantizer. The nearest neighbour projection $\Proj_{\Gamma_N} : \R \rightarrow \{ x_1^N, \dots, x_N^N \} $ induced by a Vorono\"{i} partition $\big( C_i (\Gamma_N)\big)_{i = 1, \dots, N}$ is defined by
	\begin{equation*}
		\forall \xi \in \R, \qquad \Proj_{\Gamma_N} (\xi) := \sum_{i = 1}^N x_i^N \1_{\xi \in C_i (\Gamma_N) }.
	\end{equation*}
	We can now define the quantization of $X$ by composing $\Proj_{\Gamma_N}$ and $X$
	\begin{equation*}
		\widehat X^{\Gamma_N} = \Proj_{\Gamma_N} (X) = \sum_{i = 1}^N x_i^N \1_{X \in C_i (\Gamma_N) }
	\end{equation*}
	and the point-wise error induced by the replacement of $X$ by $\widehat X^{\Gamma_N}$ given by
	\begin{equation*}
		\vert X - \widehat X^{\Gamma_N} \vert = \dist \big( X, \{ x_1^N, \dots, x_N^N \} \big) = \min_{i = 1, \dots, N} \vert X - x_i^N \vert.
	\end{equation*}
\end{definition}

In order to alleviate the notations, from now on we write $\widehat X^N$ in place of $\widehat X^{\Gamma_N}$.

\begin{definition}
	The $L^2$-mean (or mean quadratic) quantization error induced by the replacement of $X$ by the quantization of X using a $N$-quantizer $\Gamma_N \subset \R$ is defined as the quadratic norm of the point-wise error previously defined
	\begin{equation*}
		\Vert X - \widehat X^N \Vert_{_2} := \bigg( \E \Big[ \min_{i = 1, \dots, N} \vert X - x_i^N \vert^2 \Big] \bigg)^{1/2} = \bigg( \int_{\R} \min_{i = 1, \dots, N} \vert \xi - x_i^N \vert^2 \Prob_{_{X}} ( d \xi ) \bigg)^{1/2}.
	\end{equation*}

	It is convenient to define the quadratic distortion function at level $N$ as the squared mean quadratic quantization error on $(\R)^N$:
	\begin{equation*}
		\Distortion : x = (x_1^N, \dots, x_N^N ) \longmapsto \E \Big[ \min_{i = 1, \dots, N} \vert X - x_i^N \vert^2 \Big] = \Vert X - \widehat X^N \Vert_{_2}^2.
	\end{equation*}
\end{definition}

\begin{remark}
	All these definitions can be extended to the $L^p$ case. For example the $L^p$-mean quantization error induced by a quantizer of size $N$ is
	\begin{equation*}
		\Vert X - \widehat X^N \Vert_{_p} := \bigg( \E \Big[ \min_{i = 1, \dots, N} \vert X - x_i^N \vert^p \Big] \bigg)^{1/p} = \bigg( \int_{\R} \min_{i = 1, \dots, N} \vert X - x_i^N \vert^p \Prob_{_{X}} ( d \xi ) \bigg)^{1/p}.
	\end{equation*}
\end{remark}

We briefly recall some classical theoretical results, see \cite{graf2000foundations,pages2018numerical} for further details.

\begin{theorem}{(Existence of optimal N-quantizers)}\label{EB:existence}
	Let $X \in L^2 ( \Prob )$ and $N \in \Integer^*$.
	\begin{enumerate}[label=(\alph*)]
		\item The quadratic distortion function $\Distortion$ at level $N$ attains a minimum at an $N$-tuple $x^{(N)} = ( x_1^N, \dots, x_N^N )$ and $\Gamma_N = \{ x_i^{N}, i = 1, \dots, N \}$ is a quadratic optimal quantizer at level $N$.
		\item If the support of the distribution $\Prob_{_{X}}$ of $X$ has at least $N$ elements, then $x^{(N)} = ( x_1^N, \dots, x_N^N )$ has pairwise distinct components, $ \Prob_{_{X}} \big( C_i ( x^{(N)} ) \big) > 0, \,  i = 1, \dots, N$. Furthermore, the sequence $N \mapsto \inf_{x \in ( \R )^N} \Distortion(x)$ converges to $0$ and is decreasing as long as it is positive.
	\end{enumerate}
\end{theorem}

Following the existence of a minimum for $\Distortion$ at $x^{(N)}$, we can define an optimal quadratic $N$-quantizer.
\begin{definition}
	A grid associated to any $N$-tuple solution to the above distortion minimization problem is called an optimal quadratic $N$-quantizer.
\end{definition}

A really interesting and useful property concerning quadratic optimal quantizers is the stationarity property.

\begin{proposition}{(Stationarity)}
	Assume that the support of $\Prob_{_{X}}$ has at least $N$ elements. Any $L^2$-optimal $N$-quantizer $\Gamma_N \in ( \R )^N$ is stationary in the following sense: for every Vorono\"{i} quantization $\widehat X^N$ of $X$,
	\begin{equation*}
		\E \big[ X \mid \widehat X^N \big] = \widehat X^N.
	\end{equation*}
\end{proposition}

\begin{cor}\label{EB:conditionalStationarity}
	If $\widehat X^N$ is a $L^2$-optimal quantization of $X$, hence has the above stationarity property, and $f(X) \in L^2 ( \Prob )$ with $f:\R \rightarrow \R$ then
	\begin{equation*}
		\E \big[ f (\widehat X^N ) (X - \widehat X^N ) \big] = 0.
	\end{equation*}
\end{cor}

\begin{proof}
	The proof is straightforward, indeed
	\begin{equation*}
		\begin{aligned}
			\E \big[ f (\widehat X^N ) (X - \widehat X^N ) \big]
			 & = \E \Big[ \E \big[ f(\widehat X^N) (X - \widehat X^N ) \mid \widehat X^N \big] \Big] = \E \big[ f(\widehat X^N) \E [ X - \widehat X^N \mid \widehat X^N ] \big] \\
			 & = \E \Big[ f(\widehat X^N) \big( \E \big[ X \mid \widehat X^N \big] - \widehat X^N \big) \Big] = 0.
		\end{aligned}
	\end{equation*}
\end{proof}
We now take a look at the asymptotic behaviour in $N$ of the quadratic mean quantization error. We saw in Theorem \ref{EB:existence} that the infimum of the quadratic distortion converges to $0$ as $N$ goes to infinity. The next Theorem, known as Zador's Theorem, analyzes the rate of convergence of the $L^p$-mean quantization error.

\begin{theorem}{(Zador's Theorem)}\label{EB:zador} Let $p \in (0, + \infty)$.
	\begin{enumerate}[label=(\alph*)]
		\item {\sc Sharp rate}. Let $X \in L^{p+ \delta}(\Prob)$ for some $\delta > 0$. Let $\Prob_{_{X}} (d \xi) = \varphi(\xi) \cdot \lambda ( d \xi ) + \nu ( d \xi ) $, where $\nu ~ \bot ~ \lambda$ is the singular component of $\Prob_{_{X}}$ with respect to the Lebesgue measure $\lambda$ on $\R$. Then
		      \begin{equation*}
			      \lim_{N \rightarrow + \infty} N \min_{\Gamma_N \subset \R, \vert \Gamma_N \vert \leq N } \Vert X - \widehat X^N \Vert_{_p} = \widetilde{J}_{p,1} \bigg[ \int_{\R} \varphi^{\frac{1}{1+p}} d \lambda \bigg]^{1+\frac{1}{p} }
		      \end{equation*}
		      with $\widetilde{J}_{p,1} = \frac{1}{2^p (p+1)}$.

		\item {\sc Non asymptotic upper-bound}. Let $\delta > 0$. There exists a real constant $C_{1,p,\delta} \in (0, +\infty )$ such that, for every $\R$-valued random variable $X$,
		      \begin{equation*}
			      \forall N \geq 1, \qquad \min_{\Gamma_N \subset \R, \vert \Gamma_N \vert \leq N } \Vert X - \widehat X^N \Vert_{_p} \leq C_{1,p,\delta} \sigma_{\delta+p} (X) N^{- 1}
		      \end{equation*}
		      where, for $r \in (0, + \infty), \, \sigma_r(X) = \min_{a \in \R} \Vert X - a \Vert_{_r} < + \infty$.
	\end{enumerate}
\end{theorem}

Now, we state some intuitive but remarkable results concerning the local behaviour of the optimal quantizers.

\begin{lemme}\label{EB:unioncellsincompactiscompact}
	Let $\Prob_{_{X}}$ be a distribution on the real line with connected support $I_{\Prob_{_{X}}}:=\supp ( \Prob_{_{X}} )$. Let $\Gamma_N = \{ x_1^N, \dots, x_N^N \}$ be a sequence of $r$-optimal quantizers, $r>0$. Let $[a,b]$, be a closed interval then
	\begin{equation*}
		\bigcup_N \bigcup_{C_i ( \Gamma_N ) \cap [a,b] \neq \emptyset} C_i ( \Gamma_N ) \subset K_0
	\end{equation*}
	where $K_0$ is a compact set.
\end{lemme}

\begin{proof}
	First, if $+ \infty \notin \overline{I_{\Prob_{_{X}}}}$ then the upper-bound of $K_0$ is the upper-bound of $\overline{I_{\Prob_{_{X}}}}$ otherwise if $+ \infty \in \overline{I_{\Prob_{_{X}}}}$, let $b_0 \in I_{\Prob_{_{X}}}$ such that $b_0 < b$, as $\Prob_{_{X}}$ has a density, then $\Prob_{_{X}}\big( \{ b_{0} \}\big) = \Prob_{_{X}}\big( \{ b \}\big) = 0$. Considering the weighted empirical measure
	\begin{equation*}
		\Prob_{_{\widehat X^N}} := \sum_{i = 1}^{N} \Prob_{_{X}} \big( C_i ( \Gamma_N ) \big) \delta_{x_i^N} \xrightarrow{N \rightarrow +\infty } \Prob_{_{X}}
	\end{equation*}
	then $\Prob_{_{\widehat X^N}} \big( [b_0, b ] \big) \xrightarrow{N \rightarrow +\infty } \Prob_{_{X}} \big( [b_0, b ] \big) < \Prob_{_{X}} \big( [b_0, +\infty ) \big)$. Moreover, one notices that
	\begin{equation*}
		\Prob_{_{\widehat X^N}} \big( [b_0, b] \big) = \Prob_{_{X}} \left( \bigcup_{i \in \{ i_{b_0}, \dots, i_{b} \} } C_i ( \Gamma_N ) \right) = \Prob_{_{\widehat X^N}} \left( \bigcup_{i \in \{ i_{b_0}, \dots, i_{b} \} } C_i ( \Gamma_N ) \right)
	\end{equation*}
	where $x_{i_{u}}^N$ is the centroid of the cell that contains $u$. Then, as $[ b_0, x_{i_{b} + 1/2}^N ] \subset \bigcup_{i \in \{ i_{b_0}, \dots, i_{b} \} } C_i ( \Gamma_N )$
	\begin{equation*}
		\Prob_{_{X}} \big( [ b_0, x_{i_{b} + 1/2}^N ] \big) \leq \Prob_{_{\widehat X^N}} \big( [b_0, b ] \big) \xrightarrow{N \rightarrow +\infty } \Prob_{_{X}} \big( [b_0, b ] \big) < \Prob_{_{X}} \big( [b_0, +\infty ) \big)
	\end{equation*}
	hence, $\limsup_{N} x_{i_{b} + 1/2}^N < + \infty$ and $\sup_{N} x_{i_{b} + 1/2}^N < + \infty$, which gives us the upper-bound of $K_0$: $\sup_{N} x_{i_{b} + 1/2}^N$.

	Finally, if $- \infty \notin \overline{I_{\Prob_{_{X}}}}$ then the lower-bound of $K_0$ is the lower-bound of $\overline{I_{\Prob_{_{X}}}}$ otherwise if $- \infty \in \overline{I_{\Prob_{_{X}}}}$, then following the same idea as above, we can apply the same deductions in order to show that $\inf_{N} x_{i_{a}-1/2}^N > - \infty$ which gives us the lower-bound of $K_0$: $\inf_{N} x_{i_{a} - 1/2}^N$.
	In conclusion, $K_0 := \supp(\Prob_{_{X}}) \bigcap [ \inf_{N} x_{i_{a}-1/2}^N, \sup_{N} x_{i_{b} + 1/2}^N ]$.
\end{proof}

The next result, proved in \cite{localdistor}, deals with the local behaviour of optimal quantizer, more precisely it characterises the rate of convergence, in function of $N$, of the weights and the local distortions associated to an optimal quantizer. This is the key result of the first part of this paper. It allows us to extend the weak error bound of order two to less regular functions than those originally considered in \cite{pages1998space}, namely differentiable functions with Lipschitz continuous derivative.

\begin{theorem}{(Local behaviour of optimal quantizers)}\label{EB:local}
	Let $\Prob_{_{X}}$ be a distribution on the real line with connected support $\supp(\Prob_{_{X}})$. Assume that $\Prob_{_{X}}$ has a probability density function $\varphi$ which is positive and Lipschitz continuous on every compact set of the interior $(\underline m,\overline m)$ of $\supp(\Prob_{_{X}})$. Let $\Gamma_N = \{ x_1^N, \dots, x_N^N \}$ be a sequence of stationary and $L^r$ optimal quantizers, $r>0$.
	\begin{enumerate}[label=(\alph*)]
		\item The sequence of functions $(\psi_N)_{N \geq 1}$ defined by
		      \begin{equation*}
			      \psi_N(\xi) := N \sum_{i = 1}^N \1_{C_i ( \Gamma_N )}(\xi) \Prob_{_{X}} \big(C_i(\Gamma_N)\big), \quad N \geq 1,
		      \end{equation*}
		      converges uniformly on compact sets of $(\underline m,\overline m)$ towards $c_{\varphi, 1/(r+1)} \varphi^{\frac{r}{r+1}}$, with $c_{\varphi, 1/(r+1)} = \Vert \varphi \Vert_{_{1/(1+r)}}^{-1/(1+r)}$ i.e., for every $[a,b] \subset (\underline m,\overline m)$, $a<b$,
		      \begin{equation}\label{EB:cvgeCell}
			      \sup_{ \{ i : x_i^N \in [a,b] \}} \Big\vert N \Prob_{_{X}} \big( C_i ( \Gamma_N ) \big) - c_{\varphi, 1/(r+1)} \varphi^{\frac{r}{r+1}} ( x_i^N ) \Big\vert \xrightarrow{N \rightarrow +\infty } 0.
		      \end{equation}
		      The local distortion is asymptotically uniformly distributed i.e., for every $[a,b] \subset (\underline m,\overline m)$,
		      \begin{equation}\label{EB:cvgeDist}
			      \sup_{ \{ i : x_i^N \in [a,b] \}} \bigg\vert N^{r+1} \int_{C_i ( \Gamma_N )} \vert x_i^N - \xi \vert^r \Prob_{_{X}} ( d \xi ) - \frac{\Vert \varphi \Vert_{_{1/(r+1)}} }{2^r (r+1) } \bigg\vert \xrightarrow{N \rightarrow +\infty } 0.
		      \end{equation}
		\item Moreover, if $\Prob_{_{X}}$ has a compact support $[\underline m,\overline m]$ and $\varphi$ is bounded away from $0$ on the whole interval $[m, M]$, then all the above convergences hold uniformly on $[\underline m,\overline m]$.
	\end{enumerate}
\end{theorem}

The next result is a weaker version of Theorem \ref{EB:local} but it is a really useful tool when dealing with weak error induced by quantization-based cubature formulas.
\begin{cor}\label{EB:cvgeDistssmaller}
	Under the same hypothesis as in Theorem \ref{EB:local} and if $1 \leq s \leq r$, we have the following result, for every $i \in \{ 1, \dots, N \}$,
	\begin{equation*}
		\limsup_{N} N^{s+1} \int_{C_i ( \Gamma_N )} \vert x_i^N - \xi \vert^s \Prob_{_{X}} ( d \xi ) = \limsup_{N} N^{s+1} \E \big[ \vert \widehat X^{N} - X \vert^s \1_{\{ \widehat X^N = x_i^N \}} \big] < +\infty.
	\end{equation*}
\end{cor}

\begin{proof}
	If $s=1$, using Schwarz's inequality
	\begin{equation*}
		\begin{aligned}
			\int_{C_i(\Gamma_N)} \vert x_i^N - \xi \vert \Prob_{_{X}} ( d \xi )                 & \leq \bigg( \int_{C_i(\Gamma_N)} \vert x_i^N - \xi \vert^2 \Prob_{_{X}} ( d \xi ) \cdot \Prob_{_{X}} \big( C_i ( \Gamma_N ) \big) \bigg)^{\frac{1}{2}}        \\
			\iff \qquad N^2 \int_{C_i(\Gamma_N)} \vert x_i^N - \xi \vert \Prob_{_{X}} ( d \xi ) & \leq \bigg( N^3 \int_{C_i(\Gamma_N)} \vert x_i^N - \xi \vert^2 \Prob_{_{X}} ( d \xi ) \cdot N \Prob_{_{X}} \big( C_i ( \Gamma_N ) \big) \bigg)^{\frac{1}{2}}.
		\end{aligned}
	\end{equation*}
	And applying Theorem \ref{EB:local} with $\Prob_{_{X}} = \varphi \cdot \lambda$ and $r=2$, one derives
	\begin{equation*}
		\limsup_{N} N^2 \int_{C_i(\Gamma_N)} \vert x_i^N - \xi \vert \Prob_{_{X}} ( d \xi ) \leq \frac{1}{2 \sqrt{3}} \big(c_{\varphi, 1/3} \Vert \varphi \Vert_{_{1/3}} \Vert \varphi^{2/3} \Vert_{_\infty} \big)^{\frac{1}{2}} < + \infty.
	\end{equation*}
	Otherwise, for $1<s<r$, using H\"{o}lder's inequality with $p=\frac{1}{s}$ and $q=\frac{1}{1-s}$
	\begin{equation*}
		\begin{aligned}
			\int_{C_i ( \Gamma_N )} \vert x_i^N - \xi \vert^s \Prob_{_{X}} ( d \xi )
			 & \leq \bigg( \int_{C_i ( \Gamma_N )} \vert x_i^N - \xi \vert^{ps} \Prob_{_{X}} ( d \xi ) \bigg)^{1/p} \bigg( \int_{C_i ( \Gamma_N )} \Prob_{_{X}} ( d \xi ) \bigg)^{1/q} \\
			 & \leq \bigg( \int_{C_i ( \Gamma_N )} \vert x_i^N - \xi \vert \Prob_{_{X}} ( d \xi ) \bigg)^s \Big( \Prob_{_{X}} \big( C_i ( \Gamma_N ) \big) \Big)^{1-s}                 \\
			\iff N^{s+1} \int_{C_i(\Gamma_N)} \vert x_i^N - \xi \vert^s \Prob_{_{X}} ( d \xi )
			 & \leq N^{s+1} \bigg( \int_{C_i ( \Gamma_N )} \vert x_i^N - \xi \vert \Prob_{_{X}} ( d \xi ) \bigg)^s \Big( \Prob_{_{X}} \big( C_i ( \Gamma_N ) \big) \Big)^{1-s}         \\
			 & \leq \bigg( N^2 \int_{C_i ( \Gamma_N )} \vert x_i^N - \xi \vert \Prob_{_{X}} ( d \xi ) \bigg)^s \Big( N \Prob_{_{X}} \big( C_i ( \Gamma_N ) \big) \Big)^{1-s}.
		\end{aligned}
	\end{equation*}
	And using the result proved above for $s=1$ and \eqref{EB:cvgeCell}, we obtain the desired result
	\begin{equation*}
		\begin{aligned}
			\limsup_{N} N^{s+1} \int_{C_i(\Gamma_N)} \vert x_i^N - \xi \vert^s & \Prob_{_{X}} ( d \xi )                                                                                                                                                    \\
			                                                                   & \leq \limsup_{N} \bigg( N^2 \int_{C_i ( \Gamma_N )} \vert x_i^N - \xi \vert \Prob_{_{X}} ( d \xi ) \bigg)^s \Big( N \Prob_{_{X}} \big( C_i ( \Gamma_N ) \big) \Big)^{1-s} \\
			                                                                   & \leq \bigg( \frac{1}{12} \Vert \varphi \Vert_{_{1/3}} \bigg)^{\frac{s}{2}} \bigg(c_{\varphi, 1/3} \Vert \varphi^{2/3} \Vert_{_\infty} \bigg)^{1-\frac{s}{2}}              \\
			                                                                   & < + \infty.
		\end{aligned}
	\end{equation*}
\end{proof}

The following result will be useful in the last part of the paper, which is the Theorem 6 in \cite{delattre2004quantization}.
\begin{theorem}\label{EB:distortimesf}
	Let $(\Gamma_N)_{N\ge 1}$ a sequence of optimal quantizers for $\Prob_{_{X}}$. Then
	\begin{equation*}
		\lim_{N \rightarrow + \infty} N^2 \E \big[ g (\widehat X^N ) \vert X - \widehat X^N \vert^2 \big] = \mathcal{Q}_2(\Prob_{_{X}}) \int g(\xi) \Prob_{_{X}} ( d \xi )
	\end{equation*}
	for every function $g:\R \rightarrow \R$ such that $\E \big[ g (X) \big] < + \infty$, with $\mathcal{Q}_2(\Prob_{_{X}})$ the Zador's constant.
\end{theorem}

The last result we state is an answer to the following question: what can we say about the rate of convergence of $\E \big[ \vert X - \widehat X^N \vert^{2 + \beta} \big]$ knowing that $\widehat X^N$ is a quadratic optimal quantization? This problem is known as the distortion mismatch problem and has been first addressed in \cite{graf2008distortion} and the results have been extended in Theorem 4.3 of \cite{pages2018improved}.

\begin{theorem}\label{EB:LrLsdistortionmismatch}[$L^r$-$L^s$-distortion mismatch]
	Let $X:(\Omega, \A, \Prob ) \rightarrow \R$ be a random variable and let $r \in (0, + \infty)$. Assume that the distribution $\Prob_{_{X}}$ of $X$ has a non-zero absolutely continuous component with density $\varphi$, i.e. $\Prob_{_{X}} (d \xi) = \varphi(\xi) \cdot \lambda ( d \xi ) + \nu ( d \xi ) $, where $\nu ~ \bot ~ \lambda$ is the singular component of $\Prob_{_{X}}$ with respect to the Lebesgue measure $\lambda$ on $\R$ and $\varphi$ is non-identically null. Let $(\Gamma_N)_{N \geq 1}$ be a sequence of $L^r$-optimal grids.
	Let $s \in (r, r+1)$. If
	\begin{equation*}
		X \in L^{\frac{s}{1+r-s}+\delta}(\Prob)
	\end{equation*}
	for some $\delta>0$, then
	\begin{equation*}
		\limsup_N N \Vert X - \widehat X^N \Vert_{_{s}} < + \infty.
	\end{equation*}
\end{theorem}

\section{Weak Error bounds for Optimal Quantization ($d=1$)} \label{EB:section:weakerror}

Let $X \in L^2 ( \Prob )$ and $\widehat X^N$ a quadratic optimal quantizer of $X$ which takes its values in the finite grid $\Gamma_N = \{ x_1^N, \dots, x_N^N \}$ of size $N$. We consider a function $f : \R \rightarrow \R$ with $f(X) \in L^2 ( \Prob )$. One of the application of the framework developed above is the approximation of expectations of the form $ \E \big[ f ( X) \big]$. Indeed, as $\widehat X^N$ is close to $X$ in $L^2 ( \Prob )$, a natural idea is to replace $X$ by $\widehat X^N$ inside the expectation
\begin{equation*}
	\E \big[ f ( \widehat X^N ) \big] = \sum_{i=1}^N f ( x_i^N ) \Prob_{_{X}} \big( C_i ( \Gamma_N ) \big).
\end{equation*}
The above formula is referred as the quantization-based cubature formula to approximate $\E \big[ f ( X ) \big] $. Now, we need to have an idea of the error we make when doing such an approximation and what is its rate of convergence as $N$ tends to infinity? For that, we want to find the largest $\alpha \in \R$, such that the beyond limit is bounded
\begin{equation}\label{EB:weakerror}
	\lim_{N \rightarrow + \infty} N^{\alpha} \big\vert \E \big[ f(X) \big] - \E \big[ f ( \widehat X^N ) \big] \big\vert \leq C_{f, X} < + \infty.
\end{equation}

The first class of function we consider is the class of Lipschitz continuous functions, more precisely piecewise affine functions and convex Lipschitz continuous functions. Then we deal with differentiable functions with piecewise-defined derivatives.

\subsection{Piecewise affine functions}
We improve the standard rate of convergence which is of order $1$ for Lipschitz continuous functions by considering a subclass of the Lipschitz continuous functions, namely piecewise affine functions. This new result shows that the weak error induced is of order $2$ ($\alpha = 2$ in \eqref{EB:weakerror}).

\begin{lemme}\label{EB:affine}
	Assume that the distribution $\Prob_{_{X}} = \varphi \cdot \lambda $ of $X$ satisfies the conditions of Theorem \ref{EB:local}. Let $f:\R \rightarrow \R$ be a Borel function.
	\begin{enumerate}[label=(\alph*)]
		\item If $f$ is a continuous piecewise affine function with finitely many breaks of affinity, then there exists a real constant $C_{f, X} > 0$ such that
		      \begin{equation*}
			      \limsup_{N} N^2 \big\vert \E \big[ f(X) \big] - \E \big[ f ( \widehat X^N ) \big] \big\vert \leq C_{f, X} < + \infty.
		      \end{equation*}
		\item However, if $f$ is not supposed continuous but is still a piecewise affine function with finitely many breaks of affinity, then there exists a real constant $C_{f, X} > 0$ such that
		      \begin{equation*}
			      \limsup_{N} N \big\vert \E \big[ f(X) \big] - \E \big[ f ( \widehat X^N ) \big] \big\vert \leq C_{f, X} < + \infty.
		      \end{equation*}
	\end{enumerate}
\end{lemme}

\begin{proof}
	Let $I$ be a compact interval containing all the affinity breaks of $f$ denoted $a_1, \dots, a_\ell$.
	\begin{enumerate}[label=(\alph*), wide, labelwidth=!, labelindent=2pt]
		\item Let $f$ supposed to be continuous. Note that $f$ is Lipschitz continuous (with coefficient denoted $[f]_{_{Lip}} := \max_{i = 1, \dots, \ell} \vert a_i \vert$). Let $\Gamma_N = \{ x_1^{N}, \dots, x_N^{N} \}$ be an $L^2$- optimal quantizer at level $N \geq 1$.
		      \begin{align}\label{EB:diffpiecewiseaffine}
			      \E \big[ f(X) \big] - \E \big[ f ( \widehat X^N ) \big]
			       & = \sum_{i = 1}^N \int_{C_i(\Gamma_N)} \big( f(\xi) - f ( x_i^N ) \big) \Prob_{_{X}} ( d \xi ) \notag \\
			       & = \sum_{i \in J^N_f} \int_{C_i(\Gamma_N)} \big( f(\xi) - f ( x_i^N ) \big) \Prob_{_{X}} ( d \xi )
		      \end{align}
		      where $J^N_f = \{i : C_i(\Gamma_N) \textrm{ contains an affinity break}\}$ since all other terms are $0$. Indeed, as $ f(\xi) = \alpha_i \xi + \beta_i $ on $C_i(\Gamma_N)$ and using Corollary \ref{EB:conditionalStationarity}
		      \begin{equation*}
			      \int_{C_i(\Gamma_N)} \big( f(\xi) - f ( x_i^N ) \big) \Prob_{_{X}} ( d \xi ) = \alpha_i \E \big[ (X-\widehat X^N) \1_{\{ \widehat X^N = x_i^N \}} \big] = 0.
		      \end{equation*}
		      Now, taking the absolute value in \eqref{EB:diffpiecewiseaffine}, we have
		      \begin{align}\label{EB:affine:majoration}
			      \big\vert \E \big[ f(X) \big] - \E \big[ f ( \widehat X^N ) \big] \big\vert
			       & \leq \card (J_f^N ) \max_{i \in J_f^N} \int_{C_i(\Gamma_N)} \vert f(\xi) - f ( x_i^N ) \vert \Prob_{_{X}} ( d \xi ) \notag \\
			       & \leq \card (J_f^N ) [f]_{_{Lip}} \max_{i \in J_f^N} \int_{C_i(\Gamma_N)} \vert \xi - x_i^N \vert \Prob_{_{X}} ( d \xi )
		      \end{align}
		      and using Corollary \ref{EB:cvgeDistssmaller} with $s=1$, we have the desired result, with an explicit asymptotic upper bound,
		      \begin{equation*}
			      \begin{aligned}
				      \limsup_{N} N^2\big\vert \E \big[ f(X) \big] - \E \big[ f ( \widehat X^N ) \big] \big\vert
				       & \leq [f]_{_{Lip}} \lim_{N} \card (J_f^N ) \max_{i \in J_f^N} N^2 \int_{C_i(\Gamma_N)} \vert \xi - x_i^N \vert \Prob_{_{X}} ( d \xi )                \\
				       & < [f]_{_{Lip}} \frac{\ell}{2 \sqrt{3}} \big(c_{\varphi, 1/3} \Vert \varphi \Vert_{_{1/3}} \Vert \varphi^{1/3} \Vert_{_{\infty}} \big)^{\frac{1}{2}} \\
				       & < + \infty.
			      \end{aligned}
		      \end{equation*}
		\item The sum in \eqref{EB:diffpiecewiseaffine} in the discontinuous case is still true. However, the bound in \eqref{EB:affine:majoration} changes and becomes
		      \begin{equation*}
			      \begin{aligned}
				      \big\vert \E \big[ f(X) \big] - \E \big[ f ( \widehat X^N ) \big] \big\vert
				       & \leq 2 \ell \Vert f \Vert_{_{\infty, K_0}} \max_{i \in J_f^N} \Prob_{_{X}} \big( C_i(\Gamma_N) \big)
			      \end{aligned}
		      \end{equation*}
		      where $\Vert f \Vert_{_{\infty, K_0}}$ denotes the maximum of $\vert f \vert$ on $K_0$ and $K_0$ is defined as the compact appearing in Lemma \ref{EB:unioncellsincompactiscompact} stating that the union over all $N$ of all the cells where their intersection with the interval $[a_1, a_\ell]$ is non empty lies in a compact $K_0$, namely
		      \begin{equation*}
			      \bigcup_N \bigcup_{C_i(\Gamma_N) \cap [a_1, a_\ell] \neq \emptyset} C_i ( \Gamma_N ) \subset K_0.
		      \end{equation*}
		      The desired limit is obtained using Theorem \ref{EB:local}.
	\end{enumerate}
\end{proof}

\subsection{Lipschitz Convex functions}
Thanks to the previous result on piecewise-affine functions, we can extend the rate of convergence of order $2$ to a bigger class of functions: Lipschitz convex functions.

We recall that a real-valued function $f$ defined on a non-trivial interval $I \subset \R$ is convex if
\begin{equation*}
	f \big(t x + (1-t) y \big) \leq t f(x) + (1-t) f (y),
\end{equation*}
for every $t \in [0,1]$ and $x,y \in I$. If $f: I \rightarrow \R$ is supposed to be a convex function, then its right and left derivatives exist, are non-decreasing on $\mathring{I}$ and $\forall x \in \mathring{I}, \, f'_{-}(x) \leq f'_{+}(x)$. Moreover, as $f$ is supposed to be Lipschitz continuous, then $f'_{-}$ and $f'_{+}$ are bounded on $I$ by $[f]_{_{Lip}}$.
\begin{remark}
	One of the very interesting properties of convex functions when dealing with stationary quantizers follows from Jensen's inequality. Indeed, for every convex function $f : I \rightarrow \R$ such that $f(X) \in L^1 ( \Prob )$,
	\begin{equation*}
		\E \Big[ f \big( \E \big[ X \mid \widehat X^N \big] \big) \Big] \leq \E \Big[ \E \big[ f (X) \mid \widehat X^N \big] \Big]
	\end{equation*}
	so that,
	\begin{equation*}
		\E \big[ f (\widehat X^N) \big] \leq \E \big[ f(X) \big].
	\end{equation*}
	This means that the quantization-based cubature formula used to approximate $\E \big[ f ( X ) \big] $ is a lower-bound of the expectation.
\end{remark}

We present, here, a more convenient and general form of the well known Carr-Madan formula representation (see \cite{carr2001optimal}).

\begin{proposition}\label{EB:representationConvexGilles}
	Let $f: I \rightarrow \R$ be a Lipschitz convex function and let $I$ be any interval non trivial ($\neq \emptyset, \{a\}$) with endpoints $a, b \in \overline \R$. Then, there exists a unique finite non-negative Borel measure $\nu := \nu_{f}$ on $I$ such that, for every $c \in I$,
	\begin{equation*}
		\forall x \in I, \quad f(x) = f(c) + (x-c) f'_{+}(c) + \int_{[a,c] \cap I} (u-x)_+ \nu(du) + \int_{(c,b] \cap I} (x-u)_+ \nu(du).
	\end{equation*}
\end{proposition}

\begin{proof}
	Let $f:I \rightarrow \R$ be a Lipschitz convex function. We can define the non-negative finite measure $\nu := \nu_f$ on $I$ by setting
	\begin{equation*}
		\forall x, y \in I, \quad x \leq y, \quad \nu \big( ( x, y] \big) = f'_{+}(y) - f'_{+}(x).
	\end{equation*}
	The finiteness of $\nu$ is induced by the Lipschitz continuity of $f$ as the left and right derivatives are bounded by $[f]_{_{Lip}} = \max ( \Vert f'_+ \Vert_{_\infty}, \Vert f'_- \Vert_{_\infty} )$. Let $c \in I$, for every $x \geq c$, we have the following representation of $f(x)$:
	\begin{equation*}
		\begin{aligned}
			f(x) & = f(c) + \int_{c}^{x} f'_{+} (u) du                                       \\
			     & = f(c) + x f'_{+}(c) + \int_{c}^{x} \nu( (c,u ]) du                       \\
			     & = f(c) + x f'_{+}(c) + \int \int \1_{(c,x]}(u) \1_{(c,u]}(v) \nu(dv) ~ du \\
			     & = f(c) + x f'_{+}(c) + \int_{(c,x]} (x-v) du ~ \nu(dv)                    \\
			     & = f(c) + x f'_{+}(c) + \int_{(c,b] \cap I} (x-v)_+ \nu(dv)
		\end{aligned}
	\end{equation*}
	using Fubini's Theorem and noting that $\1_{(c,x]}(u) \1_{(c,u]}(v) = \1_{(c,x]}(v) \1_{[v,x]}(u) $. Similarly for $x \leq c$
	\begin{equation*}
		f(x) = f(c) + x f'_{+}(c) + \int_{[a, c] \cap I} (u-x)_+ \nu(du).
	\end{equation*}
	Then,
	\begin{equation*}
		\forall x \in \R, \quad f(x) = f(c) + x f'_{+}(c) + \int_{[a, c] \cap I} (u-x)_+ \nu(du) + \int_{(c, b] \cap I} (x-u)_+ \nu(du).
	\end{equation*}
\end{proof}

We can now use the representation of convex functions given above and extend the result concerning the weak error of order $2$ ($\alpha = 2$ in \eqref{EB:weakerror}).

\begin{proposition}
	We assume that the distribution $\Prob_{_{X}} = \varphi \cdot \lambda $ of $X$ satisfies the conditions of Theorem \ref{EB:local}. Let $I$ be any non-trivial interval and let $f: I \rightarrow \R$ be a Lipschitz convex function with second derivative $\nu$ (see Proposition \ref{EB:representationConvexGilles}).
	If $I_{\Prob_{_{X}}} \cap \supp ( \nu )$ is compact, with $I_{\Prob_{_{X}}}:=\supp ( \Prob_{_{X}} )$, then there exists a real constant $C_{f, X} > 0$ such that
	\begin{equation*}
		\limsup_{N} N^2 \big\vert \E \big[ f(X) \big] - \E \big[ f ( \widehat X^N ) \big] \big\vert \leq C_{f, X} < + \infty.
	\end{equation*}
\end{proposition}

\begin{remark}
	Assuming that $\supp (\nu)$ is compact actually means that $f$ is affine outside a compact set, namely that there exist $\alpha^{(\pm)}$ and $\beta^{(\pm)}$ such that $f(x) = \alpha^{(+)} x + \beta^{(+)}$, for $x$ large enough ($x \geq K_+$) and $f(x) = \alpha^{(-)} x + \beta^{(-)}$, for $x$ small enough ($x \leq K_-$). Therefore, this class of functions contains all classical vanilla financial payoffs: call, put, butterfly, saddle, straddle, spread, etc. Moreover, if $I_{\Prob_{_X}}$ is compact, such as in the uniform distribution, then there is no need for the hypothesis on $\nu$ and we could consider any Lipschitz convex functions we want. The hypothesis on the intersection allows us to consider more cases.
\end{remark}

\begin{proof}
	First we decompose the expectations across the Vorono\"{i} cells as follows
	\begin{equation*}
		\begin{aligned}
			\E \big[ f(X) - f(\widehat X^N) \big]
			 & = \sum_{i=1}^{N} \E \Big[ \big( f(X) - f(\widehat X^N) \big) \1_{ \{ X \in C_i (\Gamma_N) \} } \Big]         \\
			 & = \sum_{i=1}^{N} \E \Big[ \big( f(X) - f(x_i^N) \big) \1_{ \{ X \in ( x_{i-1/2}^N, x_{i+1/2}^N ] \} } \Big].
		\end{aligned}
	\end{equation*}
	We use the integral representation of the convex function $f$, of the Proposition \ref{EB:representationConvexGilles}, with $x:=X$ and $c := x_i$ and with the stationarity conditional property given by Corollary \ref{EB:conditionalStationarity}, the first term cancels out, for every $i$,
	\begin{equation*}
		\E \Big[ ( X-x_i^N ) f'_{+} ( x_i^N ) \1_{ \{ X \in C_i (\Gamma_N) \} } \Big] = 0.
	\end{equation*}
	Hence, we obtain
	\begin{align}\label{EB:expansionerrorconvex}
		\E & \Big[ \big(f(X) - f(x_i^N)\big) \1_{ \{ X \in ( x_{i-1/2}^N, x_{i+1/2}^N] \} } \Big] \notag                                                                             \\
		   & = \E \Bigg[ \bigg(\int_{[a,x_i^N] \cap I} (u-X)_+ \nu(du) + \int_{(x_i^N,b] \cap I} (X-u)_+ \nu(du) \bigg) \1_{ \{ X \in ( x_{i-1/2}^N, x_{i+1/2}^N] \} } \Bigg] \notag \\
		   & = \E \bigg[ \int_{ (x_{i-1/2}^N,x_i^N]} (u-X)_+ \nu(du) \1_{ \{ X \in ( x_{i-1/2}^N, x_i^N] \} } \bigg]                                                                 \\
		   & \qquad \qquad \qquad+ \E \bigg[ \int_{(x_i^N,x_{i+1/2}^N )} (X-u)_+ \nu(du) \1_{ \{ X \in [ x_i^N, x_{i+1/2}^N ] \} } \bigg]. \notag
	\end{align}
	The interval $(x_{i-1/2}^N,x_i^N]$ in the integral is left-open because when $u = x_{i-1/2}^N$, as $X \in ( x_{i-1/2}^N, x_i^N]$, $(u-X)_+ = 0$. The same remark can be made concerning the right open-bound of the interval $( x_i^N,x_{i+1/2}^N )$ in the integral. Now, using a crude upper-bound for \eqref{EB:expansionerrorconvex}, we get
	\begin{equation*}
		\begin{aligned}
			\E \Big[ \big(f(X) - f(x_i^N)\big) \1_{ \{ X \in ( x_{i-1/2}^N, x_{i+1/2}^N] \} } \Big]
			 & \leq \E \Big[ (x_i^N-X ) \nu \big( ( x_{i-1/2}^N, x_i^N ] \big) \1_{ \{ X \in ( x_{i-1/2}^N, x_i^N] \} } \Big]   \\
			 & \qquad + \E \Big[ (X-x_i^N) \nu \big( ( x_i^N, x_{i+1/2}^N ) \big) \1_{ \{ X \in [ x_i^N,x_{i+1/2}^N] \} } \Big] \\
			 & \leq \E \big[ \vert x_i^N-X \vert \1_{ \{ X \in C_i (\Gamma_N) \} } \big] \nu \big( C_i (\Gamma_N) \big)
		\end{aligned}
	\end{equation*}
	as $\nu \big( ( x_{i-1/2}^N, x_{i+1/2}^N ) \big) \leq \nu \big( C_i (\Gamma_N) \big)$. Hence
	\begin{equation*}
		\begin{aligned}
			0 \leq \E \big[ f(X) - f(\widehat X^N) \big]
			 & \leq \sum_{i=1}^{N} \E \big[ \vert x_i^N-X \vert \1_{ \{ X \in C_i (\Gamma_N) \} } \big] \nu \big( C_i (\Gamma_N) \big)                              \\
			 & \leq \sum_{i=1}^{N} \E \big[ \vert x_i^N-X \vert \1_{ \{ X \in C_i (\Gamma_N) \} } \big] \1_{ \{ x_i^N \in J_\nu \} } \nu \big( C_i (\Gamma_N) \big) \\
		\end{aligned}
	\end{equation*}
	with $J_\nu := [ \inf_{N} x_{i_a-1/2}^N, \sup_{N} x_{i_b+1/2}^N ]$ where $x_{i_{a}}^N$ and $x_{i_{b}}^N$ are the centroids of the optimal quantizer of size $N$ that contains, respectively, the infimum and the supremum of the support of $\nu$, denoted by $a$ and $b$, respectively. Hence, $x_{i_a-1/2}^N$ is the lower bound of the Vorono\"{i} cell $C_{i_a}(\Gamma_N)$ associated to the centroid $x_{i_a}^N$ and $x_{i_b+1/2}^N$ is the upper bound of the Vorono\"{i} cell $C_{i_b}(\Gamma_N)$ associated to the centroid $x_{i_b}^N$. If $a$ is not contained in $I_{\Prob_{_{X}}}$, then the lower bound of $J_\nu$ is set to $a$, and the same hold for $b$: if it is not contained in $I_{\Prob_{_{X}}}$, the upper bound of $J_\nu$ is set to $b$. Then,
	\begin{equation*}
		\begin{aligned}
			N^2 \E \big[ f(X) - f(\widehat X^N) \big]
			 & \leq N^2 \sum_{i=1}^{N} \E \big[ \vert x_i^N-X \vert \1_{ \{ X \in C_i (\Gamma_N) \} } \big] \1_{ \{ x_i^N \in J_\nu \} } \nu( C_i (\Gamma_N) )                               \\
			 & \leq N^2 \sup_{i: x_i^N \in I_{\Prob_{_{X}}} \cap J_\nu} \E \big[ \vert \widehat X^{N} - X \vert \1_{ \{ X \in C_i (\Gamma_N) \} } \big] \sum_{i=1}^{N} \nu( C_i (\Gamma_N) ) \\
			 & \leq \nu ( I_{\Prob_{_{X}}} ) N^2 \sup_{i: x_i^N \in I_{\Prob_{_{X}}} \cap J_\nu} \E \big[ \vert \widehat X^{N} - X \vert \1_{ \{ X \in C_i (\Gamma_N) \} } \big]             \\
		\end{aligned}
	\end{equation*}

	yielding the desired result with Theorem \ref{EB:local} if $I_{\Prob_{_{X}}} \cap J_\nu$ is compact.

	Under the hypothesis $I_{\Prob_{_{X}}} \cap ~\supp(\nu)$ compact, then by Lemma \ref{EB:unioncellsincompactiscompact},
	\begin{equation*}
		\bigcup_N \bigcup_{x_i^N \in I_{\Prob_{_{X}}} \cap ~\supp(\nu)} C_i ( \Gamma_N ) \subset \bigcup_N \bigcup_{C_i (\Gamma_N) \cap I_{\Prob_{_{X}}} \cap ~\supp(\nu) \neq \emptyset} C_i ( \Gamma_N ) \subset K_0,
	\end{equation*}
	with $K_0 := I_{\Prob_{_{X}}} \cap J_{\nu}$ compact, which is what we were looking for.
\end{proof}

\begin{proposition}
	Assume that the distribution $\Prob_{_{X}} = \varphi \cdot \lambda $ of $X$ satisfies the conditions of Theorem \ref{EB:local} not only on compact sets but uniformly. Let $I$ be any non-trivial interval then for every function $f: I \rightarrow \R$ Lipschitz convex with second derivative $\nu$ defined as in Proposition \ref{EB:representationConvexGilles}, there exists a real constant $C_{f, X} > 0$ such that
	\begin{equation*}
		\limsup_{N} N^2 \big\vert \E \big[ f(X) \big] - \E \big[ f ( \widehat X^N ) \big] \big\vert \leq C_{f, X} < + \infty.
	\end{equation*}
\end{proposition}

\begin{proof}
	This proof is exactly the same as above the Proposition.
\end{proof}

\begin{remark}
	It has not be shown yet that Gaussian or Exponential random variables satisfy the conditions of Theorem \ref{EB:local} uniformly but empirical tests tend to confirm that they exhibit the error bound property for Lipschitz convex functions. More details are given in the numerical part.
\end{remark}

\subsection{Differentiable functions}

In the following proposition, we deal with functions that are piecewise-defined and where their piecewise-defined derivatives are supposed to be locally-Lipschitz continuous or locally $\alpha$-H\"{o}lder continuous on the non-bounded parts of the interval. We define below what we mean by locally-Lipschitz and locally $\alpha$-H\"{o}lder.

\begin{definition}\label{EB:deflocliplocholder}
	\begin{itemize}
		\item A function $f:I \rightarrow \R$ is supposed to be locally-Lipschitz continuous, if
		      \begin{equation*}
			      \forall x,y \in I \quad \vert f(x) - f(y) \vert \leq [f]_{_{Lip, loc}} \vert x - y \vert \big( g(x) + g(y) \big)
		      \end{equation*}
		      where $[f]_{_{Lip, loc}}$ is a real constant and $g: \R \rightarrow \R_+$.

		\item A function $f:I \rightarrow \R$ is supposed to be locally $\alpha$-H\"{o}lder continuous, if
		      \begin{equation*}
			      \forall x,y \in I \quad \vert f(x) - f(y) \vert \leq [f]_{_{\alpha, loc}} \vert x - y \vert^{\alpha} \big( g(x) + g(y) \big)
		      \end{equation*}
		      where $[f]_{_{\alpha, loc}}$ is a real constant and $g: \R \rightarrow \R_+$.
	\end{itemize}

\end{definition}

\begin{proposition}
	Assume that the distribution $\Prob_{_{X}}$ of $X$ satisfies the conditions of the $L^r$-$L^s$-distortion mismatch Theorem \ref{EB:LrLsdistortionmismatch} and Theorem \ref{EB:local} concerning the local behaviours of optimal quantizers.
	If $f:\R \rightarrow \R$ is a piecewise-defined continuous function with finitely many breaks of affinity $\{ a_1, \dots, a_K \}$, where $ - \infty = a_0 < a_1 < \cdots < a_K < a_{K+1} = + \infty $, such that the piecewise-defined derivatives denoted $(f'_{k})_{k=0, \dots, d}$ are either
	\begin{enumerate}[label=(\alph*)]
		\item locally-Lipschitz continuous on $(a_k, a_{k+1})$ where $\exists \, q_k > 3$ such that the $q_k$-th power of $g_k : (a_k, a_{k+1}) \rightarrow \R_+$ defined in Definition \ref{EB:deflocliplocholder} are convex and $\big( \Vert g_k (X) \Vert_{_{q_k}} \big)_{k = 1, \dots, K} < + \infty$. Then there exists a real constant $C_{f, X} > 0$ such that
		      \begin{equation*}
			      \limsup_{N} N^2 \big\vert \E \big[ f(X) \big] - \E \big[ f ( \widehat X^N ) \big] \big\vert \leq C_{f, X} < + \infty.
		      \end{equation*}
		\item or locally $\alpha$-H\"{o}lder continuous on $(a_k, a_{k+1})$, $\alpha \in (0,1)$, where $\exists \, q_k> \frac{3}{2-\alpha}$ such that the $q_k$-th power of $g_k : (a_k, a_{k+1}) \rightarrow \R_+$ defined in Definition \ref{EB:deflocliplocholder} are convex and $\big( \Vert g_k (X) \Vert_{_{q_k}} \big)_{k = 1, \dots, K} < + \infty$. Then there exists a real constant $C_{f, X} > 0$ such that
		      \begin{equation*}
			      \limsup_{N} N^{1 + \alpha} \big\vert \E \big[ f(X) \big] - \E \big[ f ( \widehat X^N ) \big] \big\vert \leq C_{f, X} < + \infty.
		      \end{equation*}
	\end{enumerate}

\end{proposition}

\begin{proof}
	\begin{enumerate}[label=(\alph*), wide, labelwidth=!, labelindent=2pt]
		\item Let $\Gamma_N = \{ x_1^{N}, \dots, x_N^{N} \}$ be a $L^2$- optimal quantizer at level $N \geq 1$. In the first place, we define the set of all the indexes of the Vorono\"{i} cells that contains a break of affinity
		      \begin{equation*}
			      I_{reg}^N = \big\{ i = 1, \dots, N : C_i(\Gamma_N) \cap [a_1, a_K] \neq \emptyset \big\}.
		      \end{equation*}
		      Hence,
		      \begin{equation*}
			      \begin{aligned}
				      \E \big[ f ( \widehat X^N ) \big] - \E \big[ f(X) \big]
				       & = \underbrace{\sum_{i \in I_{reg}^N} \int_{C_i(\Gamma_N)} \big( f(x_i^N) - f(\xi)\big) \Prob_{_{X}} ( d \xi )}_{(A)}                                   \\
				       & \qquad \qquad \qquad \qquad + \underbrace{\sum_{i \notin I_{reg}^N } \int_{C_i(\Gamma_N)} \big( f(x_i^N) - f(\xi) \big) \Prob_{_{X}} ( d \xi )}_{(B)}.
			      \end{aligned}
		      \end{equation*}
		      First, we deal with the $(B)$ term. As, $i \notin I_{reg}^N$, $f$ is differentiable in $C_i(\Gamma_N)$ and admits a first-order Taylor expansion at the point $x_i^N$, moreover by Corollary \ref{EB:conditionalStationarity}, $\int_{C_i(\Gamma_N)} f' (x_i^N) ( \xi - x_i^N ) \Prob_{_{X}} (d \xi) = 0$, hence
		      \begin{equation*}
			      \int_{C_i(\Gamma_N)} \big( f(x_i^N) - f(\xi) \big) \Prob_{_{X}} ( d \xi ) = \int_{C_i(\Gamma_N)} \int_{0}^{1} \big( f' (x_i^N) - f' (t x_i^N + (1-t) \xi) \big) ( x_i^N - \xi ) dt \Prob_{_{X}} ( d \xi ).
		      \end{equation*}
		      Now, we take the absolute value and we use the locally Lipschitz property of the derivative, yielding
		      \begin{equation}\label{EB:usinglocallipcondi}
			      \begin{aligned}
				       & \bigg\vert \int_{C_i(\Gamma_N)} \big( f(x_i^N) - f(\xi) \big) \Prob_{_{X}} ( d \xi ) \bigg\vert                                                                                                  \\
				       & \qquad \qquad \leq \int_{C_i(\Gamma_N)} \int_{0}^{1} \vert f' (x_i^N) - f' (t x_i^N + (1-t) \xi ) \vert \vert x_i^N - \xi \vert dt \Prob_{_{X}} ( d \xi )                                        \\
				       & \qquad \qquad \leq [f']_{_{k, Lip, loc}} \int_{C_i(\Gamma_N)} \int_{0}^{1} (1-t) \vert x_i^N - \xi \vert^2 \big( g_{k_i}(x_i^N) + g_{k_i}(t x_i^N + (1-t) \xi ) \big) dt \Prob_{_{X}} ( d \xi ),
			      \end{aligned}
		      \end{equation}
		      with $k_i := \{ k = 0, \dots, d : x_i \in (a_k, a_{k+1}) \}$. Under the convex hypothesis of $g^{q_{k_i}}_{k_i}$, we have that
		      \begin{equation*}
			      g_{k_i}(t x_i^N + (1-t) \xi) \leq \max \big( g_{k_i}( x_i^N ), g_{k_i}( \xi ) \big) \leq g_{k_i}( x_i^N ) + g_{k_i}( \xi ),
		      \end{equation*}
		      thus
		      \begin{equation*}
			      \begin{aligned}
				      \int_{C_i(\Gamma_N)} \int_{0}^{1} (1-t) \vert x_i^N - \xi \vert^2
				       & \big( g_{k}(x_i^N) + g_k(t x_i^N + (1-t) \xi) \big) dt \Prob_{_{X}} ( d \xi )                                                   \\
				       & \leq \frac{1}{2} \int_{C_i(\Gamma_N)} \vert x_i^N - \xi \vert^2 \big( 2 g_{k}(x_i^N) + g_{k}(\xi) \big) \Prob_{_{X}} ( d \xi ).
			      \end{aligned}
		      \end{equation*}
		      Now, taking the sum over all $i \notin I_{reg}^N$ and denoting $[f']_{_{Lip, loc}} := \max_{k} [f']_{_{k, Lip, loc}}$
		      \begin{equation}\label{EB:boundp2}
			      \begin{aligned}
				      \big\vert (B) \big\vert
				       & \leq \frac{1}{2} [f']_{_{Lip, loc}} \sum_{i \notin I_{reg}^N } \int_{C_i(\Gamma_N)} \vert x_i^N - \xi \vert^2 \big( 2 g_{k_i}(x_i^N) + g_{k_i}(\xi) \big) \Prob_{_{X}} ( d \xi ) \\
				       & \leq \frac{K}{2} [f']_{_{Lip, loc}} \max_{k} \E \Big[ \vert \widehat X^N - X \vert^2 \big( 2 g_{k}(\widehat X^N) + g_{k}(X) \big) \Big]                                          \\
				       & \leq \frac{K}{2} [f']_{_{Lip, loc}} \max_{k} \Vert \widehat X^N - X \Vert^2_{_{2 p_k}} \big( 2 \Vert g_{k}(\widehat X^N) \Vert_{_{q_k}} + \Vert g_{k}(X) \Vert_{_{q_k}} \big)    \\
				       & \leq \frac{K}{2} [f']_{_{Lip, loc}} \Vert \widehat X^N - X \Vert^2_{_{2 p}} \max_{k} \big( 2 \Vert g_{k}(\widehat X^N) \Vert_{_{q_k}} + \Vert g_{k}(X) \Vert_{_{q_k}} \big)      \\
				       & \leq \frac{3 K}{2} [f']_{_{Lip, loc}} \Vert \widehat X^N - X \Vert^2_{_{2 p}} \max_{k} \Vert g_{k}(X) \Vert_{_{q_k}}                                                             \\
			      \end{aligned}
		      \end{equation}
		      using H\"{o}lder inequality, such that $\frac{1}{p_k} + \frac{1}{q_k} < 1$ and the convexity of $g^{q_k}$. Under the hypothesis $q_k > 3$, $p_k$ has to be in contained in the interval $(1,3/2)$, hence $p$ is defined as $p := \max_k p_k$ and using the non-decreasing property of the $L^p$ norm, we obtain the fourth inequality in \eqref{EB:boundp2}. Now, if we use the $L^r$-$L^s$-distortion mismatch Theorem \ref{EB:LrLsdistortionmismatch} with $r = 2$ and $s = 2p <3$ under the condition $X \in L^{\frac{2p}{3-2p}+\delta}(\Prob)$, we have
		      \begin{equation}\label{EB:resforp2}
			      \begin{aligned}
				      N^2 \big\vert (B) \big\vert
				       & \leq N^2 \frac{3 K}{2} [f']_{_{Lip, loc}} \Vert \widehat X^N - X \Vert^2_{_{2p}} \max_{k} \Vert g_{k}(X) \Vert_{_{q_k}} \\
				       & \qquad \xrightarrow{N \rightarrow +\infty } C_2 < + \infty.
			      \end{aligned}
		      \end{equation}
		      Secondly, we take care of the $(A)$ term. Using Lemma \ref{EB:unioncellsincompactiscompact} stating that the union over all $N$ of all the cells where their intersection with the interval $[a_1, a_K]$ is non empty lies in a compact $K_0$, namely
		      \begin{equation*}
			      \bigcup_N \bigcup_{C_i(\Gamma_N) \cap [a_1, a_K] \neq \emptyset} C_i ( \Gamma_N ) \subset K_0
		      \end{equation*}
		      and using that $f'$ is bounded on $K_0$ by $[f']_{_{Lip, K_0}}$, we can use the following integral representation of $f$
		      \begin{equation*}
			      f(x) = \int_{0}^x f'(u) du + f(0)
		      \end{equation*}
		      and the stationarity property of the optimal quantizer on $C_i(\Gamma_N)$, yielding
		      \begin{equation*}
			      \begin{aligned}
				      \bigg\vert \int_{C_i(\Gamma_N)} \big( f(x_i^N) - f(\xi) \big) \Prob_{_{X}} ( d \xi ) \bigg\vert
				       & = \bigg\vert \int_{C_i(\Gamma_N)} \int_{\xi}^{x_i^N} f'(u) du \Prob_{_{X}} ( d \xi ) \bigg\vert \\
				       & \leq [f']_{_{Lip, K_0}} \int_{C_i(\Gamma_N)} \vert \xi - x_i^N \vert \Prob_{_{X}} ( d \xi ).
			      \end{aligned}
		      \end{equation*}
		      Now, we sum among all $i \in I_{reg}^N$
		      \begin{equation*}
			      \big\vert (A) \big\vert \leq [f']_{_{Lip, K_0}} \sum_{i \in I_{reg}^N} \int_{C_i(\Gamma_N)} \vert \xi - x_i^N \vert \Prob_{_{X}} ( d \xi ).
		      \end{equation*}
		      Hence, using the result concerning the local behaviour of optimal quantizers Corollary \ref{EB:cvgeDistssmaller} as $[a_1, a_K]$ is compact, we have
		      \begin{align}\label{EB:resforp1_loclip}
			      N^2 \big\vert (A) \big\vert
			       & \leq N^2 [f']_{_{Lip, K_0}} \sum_{i \in I_{reg}^N} \int_{C_i(\Gamma_N)} \vert \xi - x_i^N \vert \Prob_{_{X}} ( d \xi ) \notag     \\
			       & \leq N^2 K [f']_{_{Lip, K_0}} \sup_{i : x_i^N \in K_0} \int_{C_i(\Gamma_N)} \vert \xi - x_i^N \vert \Prob_{_{X}} ( d \xi ) \notag \\
			       & \qquad \xrightarrow{N \rightarrow +\infty } C_1 < + \infty.
		      \end{align}
		      Finally, using \eqref{EB:resforp1_loclip} and \eqref{EB:resforp2}, we have the desired result
		      \begin{equation*}
			      N^2 \big\vert \E \big[ f(X) \big] - \E \big[ f ( \widehat X^N ) \big] \big\vert \leq N^2 \Big( \big\vert (A) \big\vert + \big\vert (B) \big\vert \Big) \xrightarrow{N \rightarrow +\infty } C_1 + C_2 < + \infty.
		      \end{equation*}

		\item When the piecewise-defined derivatives are locally $\alpha$-H\"{o}lder continuous on $(-\infty, a_1]$ and $[a_K, +\infty)$, $\alpha \in (0,1)$, the proof is very close to the locally Lipschitz case. Indeed, the first difference is in \eqref{EB:usinglocallipcondi}, where the $\vert x_i^N - \xi \vert^2 $ is replaced by $\vert x_i^N - \xi \vert^{1 + \alpha}$ and the constant is the one of the locally $\alpha$-H\"{o}lder hypothesis. This implies that \eqref{EB:boundp2} is replaced by
		      \begin{equation*}
			      \vert (B) \vert \leq \frac{3 K [f']_{_{Hol, loc}}}{2} \Vert \widehat X^N - X \Vert^{1+\alpha}_{_{(1+\alpha)p}} \max_{k} \Vert g_{k}(X) \Vert_{_{q_k}}.
		      \end{equation*}
		      Finally, using the $L^r$-$L^s$-distortion mismatch Theorem \ref{EB:LrLsdistortionmismatch} with $r = 2$ and $s = (1+\alpha)p <3$ under the condition $X \in L^{\frac{(1+\alpha)p}{3-(1+\alpha)p}+\delta}(\Prob)$, we have
		      \begin{equation*}
			      \begin{aligned}
				      N^{1+\alpha} \vert (B) \vert
				       & \leq N^{1+\alpha} \frac{3 K [f']_{_{Hol, loc}}}{2} \Vert \widehat X^N - X \Vert^{1+\alpha}_{_{(1+\alpha)p}} \max_{k} \Vert g_{k}(X) \Vert_{_{q_k}} \\
				       & \qquad \qquad \xrightarrow{N \rightarrow +\infty } C_3 < + \infty.
			      \end{aligned}
		      \end{equation*}
		      The other parts of the proof are identical, yielding the desired result.
	\end{enumerate}
\end{proof}

\begin{remark}
	If one strengthens the hypothesis concerning the piecewise locally Lipschitz continuous derivative and considers in place that the derivative is piecewise Lipschitz continuous, then the hypothesis that $X$ should satisfy the conditions of Theorem \ref{EB:LrLsdistortionmismatch} can be relaxed. Indeed, the term $\frac{3 K}{2} [f']_{_{Lip, loc}} \Vert \widehat X^N - X \Vert^2_{_{2 p}} \max_{k} \Vert g_{k}(X) \Vert_{_{q_k}}$ in \eqref{EB:boundp2} would become $\frac{1}{2} [f']_{_{Lip}} \Vert \widehat X^N - X \Vert^2_{_{2}}$ and we would conclude using Zador's Theorem \ref{EB:zador}.
\end{remark}

\section{Weak Error and Richardson-Romberg Extrapolation} \label{EB:section:RR}

One can improve the previous speeds of convergence using Richardson-Romberg extrapolation method. The Richardson extrapolation is a method that was originally introduced in numerical analysis by Richardson in 1911 (see \cite{richardson1911ix}) and developed later by Romberg in 1955 (see \cite{romberg1955vereinfachte}) whose aim was to speed-up the rate of convergence of a sequence, to accelerate the research of a solution of an ODE's or to approximate more precisely integrals.

\cite{talay1990romberg} and \cite{pages2007multi,pages2018numerical} used this concept for the computation of the expectation $\E \big[ f(X_T) \big]$ of a diffusion $(X_t)_{t \in [0, T]}$ that cannot be simulated exactly at a given time $T$ but can be approximated by a simulable process $\widetilde{X}_T^{(h)}$ using a Euler scheme with time step $h = T/n$ and $n$ the number of time step. The main idea is to use the weak error expansion of the approximation in order to highlight the term we would \textit{kill}. For example, using the following weak time discretization error of order $1$
\begin{equation*}
	\E \big[ f (X_T) \big] = \E \big[ f ( \widetilde{X}_T^{(h)}) \big] + \frac{c_1}{n} + O (n^{-2}),
\end{equation*}
one reduces the error of the approximation using a linear combination of the approximating process $\widetilde{X}_T^{(h)}$ and a refiner process $\widetilde{X}_T^{(h/2)}$, namely
\begin{equation*}
	\E \big[ f (X_T) \big] = \E \big[ 2 f ( \widetilde{X}_T^{(h/2)}) - f ( \widetilde{X}_T^{(h)}) \big] - \frac{1}{2} \frac{c_2}{n^2} + O (n^{-2}).
\end{equation*}

Our goal within the optimal quantization framework is to improve the speed of convergence of the cubature formula using the same ideas. Let us consider a random variable $X : (\Omega, \A, \Prob) \rightarrow \R $ and a quadratic-optimal quantizer $\widehat X^N$ of $X$. In our case we show that, if we are in dimension one there exists, for some functions $f$, a \textit{weak error expansion} of the form:
\begin{equation*}
	\E \big[ f (X) \big] = \E \big[ f ( \widehat X^N ) \big] + \frac{c_2}{N^2} + O (N^{-(2 + \beta)})
\end{equation*}
with $\beta \in (0, 1)$. We present in Section \ref{EB:subsection:higherdim} a similar result in higher dimension.

\subsection{In dimension one}

This first result is focused on function $f:\R \rightarrow \R$ with Lipschitz continuous second derivative. In that case, we have a \textit{weak error quantization} of order two. The first term of the expansion is equal to zero, thanks to the stationarity of the quadratic optimal quantizer.

\begin{proposition}
	Let $f:\R \rightarrow \R$ be a twice differentiable function with Lipschitz continuous second derivative. Let $X : (\Omega, \A, \Prob) \rightarrow \R $ be a random variable and the distribution of $\Prob_{_{X}}$ of $X$ has a non-zero absolutely continuous density $\varphi$ and, for every $N\geq 1$, let $\Gamma_N$ be an optimal quantizer at level $N\geq 1$ for $X$. Then, $\forall \, \beta \in (0, 1)$, we have the following expansion
	\begin{equation*}
		\E \big[ f(X) \big] = \E \big[ f(\widehat X^N) \big] + \frac{c_2}{N^2} + O ( N^{-(2+\beta)} ).
	\end{equation*}
	Moreover, if $\varphi : [a,b] \rightarrow \R_+$ is a Lipschitz continuous probability density function, bounded away from $0$ on $[a,b]$ then we can choose $\beta = 1$, yielding
	\begin{equation*}
		\E \big[ f(X) \big] = \E \big[ f(\widehat X^N) \big] + \frac{c_2}{N^2} + O ( N^{-3} ).
	\end{equation*}
\end{proposition}

\begin{proof}
	If $f$ is twice differentiable with Lipschitz continuous second derivatives, we have the following expansion
	\begin{equation*}
		f(x) = f(y) + f'(y)(x-y) + \frac{1}{2} f'' (y) (x-y)^2 + \int_{0}^{1} (1-t) \big( f'' (tx + (1-t)y) - f'' (y) \big) (x-y)^2 dt
	\end{equation*}
	hence replacing $x$ and $y$ by $X$ and $\widehat X^N$ respectively and taking the expectation yields
	\begin{equation*}
		\E \big[ f(X) \big] = \E \big[ f(\widehat X^N) \big] + \frac{1}{2} \E \big[ f'' ( \widehat X^N) \vert X - \widehat X^N \vert^2 \big] + R (X, \widehat X^N)
	\end{equation*}
	where $R (X, \widehat X) = \int_{0}^{1} (1-t) \E \big[ \big( f'' (tX + (1-t)\widehat X) - f'' (\widehat X) \big) \vert X-\widehat X \vert^2 \big] dt$.

	First, using Theorem \ref{EB:distortimesf} with $ f'' $, we have the following limit
	\begin{equation*}
		\lim_{N \rightarrow + \infty} N^2 \E \big[ f'' (\widehat X^N ) \vert X - \widehat X^N \vert^2 \big] = \mathcal{Q}_2(\Prob_{_{X}}) \int f'' (\xi) \Prob_{_{X}} ( d \xi ),
	\end{equation*}
	hence
	\begin{equation*}
		\E \big[ f(X) \big] = \E \big[ f(\widehat X^N) \big] + \frac{c_2}{N^2} + R (X, \widehat X^N).
	\end{equation*}
	Now, we look closely at asymptotic behaviour of $R (X, \widehat X^N)$. One notices that, if we consider a Lipschitz continuous function $g:\R \rightarrow \R$, for any fixed $\alpha\!\in (0,1)$,
	\begin{equation*}
		\forall\, x,\, y \!\in \R,\quad \vert g(x) - g(y)\vert \leq 2 \Vert g \Vert_{_{\infty}}^\alpha [g]_{_{Lip}}^{1-\alpha} \vert x - y \vert^{1 - \alpha}.
	\end{equation*}
	In our case, taking $g \equiv f''$, we have
	\begin{equation*}
		\begin{aligned}
			 & \E \Big[ \big( f'' (tX + (1-t)\widehat X^N) - f'' (\widehat X^N) \big) \vert X-\widehat X^N \vert^2 \Big]                                                                                                               \\
			 & \qquad \qquad \qquad \qquad \qquad \qquad \qquad \leq \E \Big[ 2 \Vert f'' \Vert_{_{\infty}}^\alpha [f'']_{_{Lip}}^{1-\alpha} t^{1-\alpha} \vert X - \widehat X^N \vert^{1 - \alpha} \vert X-\widehat X^N \vert^2 \Big] \\
			 & \qquad \qquad \qquad \qquad \qquad \qquad \qquad \leq C_{\beta, f''} t^{\beta} \E \big[ \vert X - \widehat X^N \vert^{2 + \beta} \big]
		\end{aligned}
	\end{equation*}
	with $0 < \beta < 1$ where $\beta = 1 - \alpha$, hence
	\begin{equation*}
		R (X, \widehat X^N) \leq \widetilde{C}_{\beta, f''} \E \big[ \vert X - \widehat X^N \vert^{2 + \beta} \big],
	\end{equation*}
	with $\widetilde{C}_{\beta, f''} = C_{\beta, f''} \frac{1}{(2+\beta) (1+\beta)}$. Using now Theorem \ref{EB:LrLsdistortionmismatch} with $r = 2$ and $s = 2+\beta$, we have the desired result: $ \E \big[ \vert X - \widehat X^N \vert^{2 + \beta} \big] = O ( N^{ -(2+\beta) } ) $ and finally
	\begin{equation*}
		\E \big[ f(X) \big] = \E \big[ f(\widehat X^N) \big] + \frac{c_2}{N^2} + O ( N^{-(2+\beta)} ),
	\end{equation*}
	for every $\beta \in (0,1)$.
	If moreover, the density $\varphi$ of $X$ is Lipschitz continuous, bounded away from $0$ on $[a,b]$ then we can take $\beta = 1$.

\end{proof}

Now, following the Richardson-Romberg idea, we could combine approximations with optimal quantizers $\widehat X^N$ of size $N$ and $\widehat X^{\widetilde N}$ of size $\widetilde N$, with $\widetilde N>N$ in order to \textit{kill} the residual term, leading
\begin{equation} \label{EB:RR_quantif}
	\E \big[ f ( X ) \big] = \E \Bigg[ \frac{\widetilde N^2 f (\widehat X^{\widetilde N}) - N^2 f (\widehat X^N)}{\widetilde N^2 - N^2} \Bigg] + O ( N^{-(2+\beta)} ).
\end{equation}

\begin{remark}
    For the choice of $\widetilde N$, we consider $\widetilde N := k \times N$. A natural choice for $k$ could be $k = 2$ or $ k = \sqrt 2$ but note that the complexity is proportional to $(k+1)N$. In practice it is therefore preferable to take a small $k$ that does not increase complexity too much. For the numerical example, we choose $\widetilde N := k \times N$ with $k = 1.2$, this is arbitrary and probably not optimal, however even with this $k$, we attain a weak error of order $3$.
\end{remark}

\subsection{A first extension in higher dimension} \label{EB:subsection:higherdim}

In this part, we give a first result on higher dimension concerning the weak error expansion of $\E \big[ f(X) \big]$ when approximated by $\E \big[ f(\widehat X^{N}) \big]$. In the next part, we use the following matrix norm: let $M \in \R^{d\times d}$, then $\vertiii{M} := \sup_{u:\vert u \vert=1} \vert u^{T} M u \vert$.
\begin{proposition}
	Let $f:\R^d \rightarrow \R$ be a twice differentiable function with a bounded and Lipschtiz Hessian $H$, namely $\forall x,y \in \R^d, \, \vertiii{H(x) - H(y)} \leq [H]_{_{Lip}} \vert x - y \vert$. Let $X : (\Omega, \A, \Prob) \rightarrow \R^d $ be a random vector with independent components $(X_k)_{k = 1, \dots, d}$. For every $(N_k)_{k = 1, \dots, d} \geq 1$, let $(\widehat X^{N_d}_d)_{k = 1, \dots, d}$ be quadratic optimal quantizers of $(X_k)_{k = 1, \dots, d}$ taking values in the grids $(\Gamma_{N_k})_{k = 1, \dots, d}$ respectively and we define $\widehat X^{N}$ as the product quantizer $X$ taking values in the finite grid $\Gamma_N := \bigotimes_{k=1, \dots, d} \Gamma_{N_d}$ of size $N := N_1 \times \cdots \times N_d$. Then, we have the following expansion
	\begin{equation*}
		\E \big[ f(X) \big] = \E \big[ f(\widehat X^{N}) \big] + \sum_{k=1}^{d} \frac{c_k}{N_k^2} + O \bigg( \Big( \min_{k=1:d} N_k \Big)^{-(2+\beta)} \bigg).
	\end{equation*}
\end{proposition}

\begin{proof}
	If $f$ is twice differentiable, hence we have the following Taylor's expansion
	\begin{equation*}
		\begin{aligned}
			f(x)
			 & = f(a) + \nabla f (a) (x-a) + \frac{1}{2} H (a) \cdot (x-a)^{\otimes 2}                             \\
			 & \qquad \qquad + \int_{0}^{1} (1-t) \big( H (t x + (1-t) a) - H (a) \big) \cdot (x-a)^{\otimes 2} dt
		\end{aligned}
	\end{equation*}
	where the notation $f(x,a) \cdot (x-a)^{\otimes 2}$ stands for $(x-a)^{T} f(x,a) (x-a)$.
	Replacing $x$ and $a$ by $X$ and $\widehat X^N$ respectively and taking the expectation
	\begin{equation*}
		\begin{aligned}
			\E \big[ f(X) \big]
			 & = \E \big[ f(\widehat X^N) \big] + \E \big[ \nabla f (\widehat X^N) (X-\widehat X^N) \big] + \frac{1}{2} \E \big[ H (\widehat X^N) \cdot (X-\widehat X^N)^{\otimes 2} \big] \\
			 & \qquad \qquad + \int_{0}^{1} (1-t) \E \Big[ \big( H (t X + (1-t) \widehat X^N) - H (\widehat X^N) \big) \cdot (X-\widehat X^N)^{\otimes 2} \Big] dt.
		\end{aligned}
	\end{equation*}
	Noticing that, by Corollary \ref{EB:conditionalStationarity},
	\begin{equation*}
		\begin{aligned}
			\E \big[ \nabla f (\widehat X^N) (X-\widehat X^N) \big]
			 & = \sum_{k=1}^{d} \E \bigg[ \frac{\partial f}{\partial x_k} (\widehat X^N) (X_k-\widehat X^{N_k}_k) \bigg]                                       \\
			 & = \sum_{k=1}^{d} \E \Bigg[ \E \bigg[ \frac{\partial f}{\partial x_k} (\widehat X^N) (X_k-\widehat X^{N_k}_k) \mid \widehat X_{-k} \bigg] \Bigg] \\
			 & = 0.
		\end{aligned}
	\end{equation*}
	where $\widehat X_{-k}$ denotes $(\widehat X_1^{N_1}, \dots, \widehat X_{k-1}^{N_{k-1}}, \widehat X_{k+1}^{N_{k+1}}, \dots, \widehat X_d^{N_d} )$. Hence
	\begin{equation}\label{EB:step4}
		\begin{aligned}
			\E \big[ f(X) \big]
			 & = \E \big[ f(\widehat X^N) \big] + \frac{1}{2} \E \big[ H (\widehat X^N) \cdot (X-\widehat X^N)^{\otimes 2} \big]                                   \\
			 & \qquad \qquad + \int_{0}^{1} (1-t) \E \Big[ \big( H (t X + (1-t) \widehat X^N) - H (\widehat X^N) \big) \cdot (X-\widehat X^N)^{\otimes 2} \Big] dt
		\end{aligned}
	\end{equation}
	and looking at the second term in \eqref{EB:step4}
	\begin{equation*}
		\begin{aligned}
			\E & \big[ H (\widehat X^N) \cdot (X-\widehat X^N)^{\otimes 2} \big]                                                                                                                                                                                                                      \\
			   & \qquad = \sum_{k=1}^{d} \E \bigg[ \frac{\partial^2 f}{\partial x_k^2} (\widehat X^N) \vert X_k-\widehat X^{N_k}_k \vert^2 \bigg] + 2 \sum_{k \neq l} \E \bigg[ \frac{\partial^2 f}{\partial x_k \partial x_l} (\widehat X^N) (X_k-\widehat X^{N_k}_k)(X_l-\widehat X^{N_l}_l) \bigg] \\
			   & \qquad = \sum_{k=1}^{d} \E \Bigg[ \E \bigg[\frac{\partial^2 f}{\partial x_k^2} (\widehat X^N) \vert X_k-\widehat X^{N_k}_k \vert^2 \mid \widehat X_{-k} \bigg] \Bigg]                                                                                                                \\
			   & \qquad \qquad \qquad + 2 \sum_{k \neq l} \E \Bigg[ \underbrace{\E \bigg[ \frac{\partial^2 f}{\partial x_k \partial x_l} (\widehat X^N) (X_k-\widehat X^{N_k}_k) \mid X_l \bigg]}_{=0} (X_l-\widehat X^{N_l}_l) \Bigg]                                                                \\
			   & \qquad = \sum_{k=1}^{d} \E \Bigg[ \E \bigg[\frac{\partial^2 f}{\partial x_k^2} (\widehat X^N) \vert X_k-\widehat X^{N_k}_k \vert^2 \mid \widehat X_{-k} \bigg] \Bigg]                                                                                                                \\
			   & \qquad = \sum_{k=1}^{d} \E \Bigg[ \E \bigg[ \frac{\partial^2 f}{\partial x_k^2} (x_1, \dots,x_{k-1}, \widehat X^{N_k}_k, x_{k+1}, \dots, x_d) \vert X_k-\widehat X^{N_k}_k \vert^2 \bigg] \bigg\vert_{\widehat X_{-k} = x_{-k}} \Bigg]                                               \\
			   & \qquad = \sum_{k=1}^{d} \E \Big[ \E \big[ g_{k, x_{-k}} (\widehat X^{N_k}_k) \vert X_k-\widehat X^{N_k}_k \vert^2 \big] \big\vert_{\widehat X_{-k} = x_{-k}} \Big].
		\end{aligned}
	\end{equation*}
	Now, using Theorem \ref{EB:distortimesf}, we have the following limits, for each $k$
	\begin{equation*}
		\lim_{N_k \rightarrow + \infty} N_k^2 \E \big[ g_{k, x_{-k}} (\widehat X^{N_k}_k) \vert X_k-\widehat X^{N_k}_k \vert^2 \big] = \mathcal{Q}_2(\Prob_{_{X_k}}) \int g_{k, x_{-k}} (\xi) \Prob_{_{X}} ( d \xi ).
	\end{equation*}
	Giving us the first part of the desired result
	\begin{equation*}
		\E \big[ f(X) \big] = \E \big[ f(\widehat X^{N}) \big] + \sum_{k=1}^{d} \frac{c_k}{N_k^2} + \int_{0}^{1} (1-t) \E \Big[ \big( H (t X + (1-t) \widehat X^N) - H (\widehat X^N) \big) \cdot (X-\widehat X^N)^{\otimes 2} \Big] dt
	\end{equation*}
	with $c_k := \frac{1}{2} \mathcal{Q}_2(\Prob_{_{X_k}}) \int \int g_{k, x_{-k}} (x) \Prob_{_{X_k}}(dx) \Prob_{_{X_{-k}}}(dy)$. Now, we take care of the integral part, we proceed using the same methodology as in the one dimensional case, using the hypothesis on the Hessian
	\begin{equation*}
		\E \Big[ \big\vert \big( H (t X + (1-t) \widehat X^N) - H (\widehat X^N) \big) \cdot (X-\widehat X^N)^{\otimes 2} \big\vert \Big] \leq 2 t^{\beta} [H]_{_{Lip}}^{\beta} \vertiii{H}^{1-\beta}_{_{\infty}} \E \big[ \vert X-\widehat X^N \vert^{2 + \beta} \big]
	\end{equation*}
	with $\beta \in (0,1)$ and $\vertiii{H}_{_{\infty}} := \sup_{x \in \R^d} \vertiii{H(x)}$. Hence
	\begin{equation*}
		\begin{aligned}
			\int_{0}^{1} (1-t) & \E \Big[ \big( H (t X + (1-t) \widehat X^N) - H (\widehat X^N) \big) \cdot (X-\widehat X^N)^{\otimes 2} \Big] dt \\
			                   & \leq \frac{1}{(2+\beta) (1+\beta)} C_{H,X} \E \big[ \vert X-\widehat X^N \vert^{2 + \beta} \big].
		\end{aligned}
	\end{equation*}
	Using now Theorem \ref{EB:LrLsdistortionmismatch}, let $s = 2+\beta$, we have the desired result: $\E \big[ \vert X_k - \widehat X^{N_k}_k \vert^{2 + \beta} \big] = O (N_k^{-(2+\beta)} )$ and finally
	\begin{equation*}
		\E \big[ f(X) \big] = \E \big[ f(\widehat X^{N}) \big] + \sum_{k=1}^{d} \frac{c_k}{N_k^2} + O \bigg( \Big( \min_{k=1:d} N_k \Big)^{-(2+\beta)} \bigg),
	\end{equation*}
	for every $\beta \in (0,1)$.
	If moreover, the densities $\varphi_k$ of $X_k$, for all $k=1, \dots, k$, are Lipschitz continuous, bounded away from $0$ on $[a,b]$ then we can take $\beta = 1$.

\end{proof}

\begin{remark}
	Even-though, we could be interested by considering non-independent components $(X_k)_{k=1, \dots, d}$, the independence hypothesis on the components is necessary in the proof because we proceed component by component. For example the first order term of the expansion would not be null by stationarity if the components are not independent.
\end{remark}

\section{Applications}

\subsection{Quantized Control Variates in Monte Carlo simulations}

Let $Z \in L^2(\Prob)$ be a random vector with components $(Z_k)_{k=1, \dots, d}$, we assume that we have a closed-form for $\E [ Z_k ]$, $k=1,\ldots,d$, and $f : \R^d \rightarrow \R$ our function of interest. We are interested in the quantity
\begin{equation}\label{EB:expectationofinterest}
	I := \E \big[ f (Z) \big].
\end{equation}
The standard method for approximating \eqref{EB:expectationofinterest} if we are able to simulate independent copies of $Z$ is to devise a Monte Carlo estimator. In this part, we present a reduction variance method based on quantized control variates. Let $\Xi_N$ our $d$ dimensional control variate
\begin{equation*}
	\Xi^{N} := (\Xi_k^{N} )_{k = 1, \dots, d}
\end{equation*}
where each component $\Xi_k^N$ is defined by
\begin{equation*}
	\Xi_k^N := f_k(Z_k) - \E \big[ f_k(\widehat{Z}^N_k) \big],
\end{equation*}
with $ f_k(z) := f ( \E [ Z_1 ], \dots, \E [ Z_{k-1} ], z, \E [ Z_{k+1} ], \dots, \E [ Z_d ] ) $ and $ \widehat{Z}_k^N $ is an optimal quantizer of cardinality $ N $ of the component $ Z_k $. One notices that the complexity for the evaluation of $ f_k $ is the same as the one of $ f $. Now, defining $ X^{\lambda} := f(Z) - \langle \lambda, \Xi^N \rangle $ where $ \lambda \in \R^d $, we can introduce $ I^{\lambda, N} $ as an approximation for \eqref{EB:expectationofinterest}
\begin{equation}\label{EB:approxbbiased}
	\begin{aligned}
		I^{\lambda, N}
		 & := \E \big[ X^{\lambda} \big]                                                                                              \\
		 & = \E \big[ f(Z) - \langle \lambda, \Xi^{N} \rangle \big]                                                                   \\
		 & = \E \Bigg[ f(Z) - \sum_{k=1}^{d} \lambda_k f_k(Z_k) \Bigg] + \sum_{k=1}^{d} \lambda_k \E \big[ f_k(\widehat Z_k^N) \big].
	\end{aligned}
\end{equation}
The terms $\E \big[ f_k(\widehat{Z}^N_k) \big]$ in \eqref{EB:approxbbiased} can be computed easily using the quantization-based cubature formula if we known the grids of the quantizers $(\widehat{Z}_k^N)_{k = 1, \dots, d}$ and their associated weights.
\begin{remark}
	We look for the $\lambda_{\min}$ minimizing the variance of $X^{\lambda}$
	\begin{equation*}
		\V (X^{\lambda_{\min}}) = \min \big\{ \V \big( f(Z) - \langle \lambda, \Xi^{N} \rangle \big), \lambda \in \R^d \big\}.
	\end{equation*}
	The solution of the above optimization problem is the solution of following system
	\begin{equation*}
		D (Z) \cdot \lambda = B
	\end{equation*}
	where $D(Z)$, the covariance-variance matrix of $\big(f_{k}(Z_k) \big)_{k = 1, \dots, d}$, and $B$ are given by
	\begin{equation*}
		D (Z) =
		\begin{pmatrix}
			\V \big(f_1(Z_1) \big)             & \cdots & \Cov \big(f_1(Z_1), f_d(Z_d) \big) \\
			\vdots                             & \ddots & \vdots                             \\
			\Cov \big(f_d(Z_d), f_1(Z_1) \big) & \cdots & \V \big(f_d(Z_d) \big)             \\
		\end{pmatrix}
		, \quad
		B =
		\begin{pmatrix}
			\Cov \big(f(Z), f_1(Z_1) \big) \\
			\vdots                         \\
			\Cov \big(f(Z), f_d(Z_d) \big) \\
		\end{pmatrix}
		.
	\end{equation*}
	The solution to this optimization problem can easily be solved numerically using any library of linear algebra able to solve linear systems thanks to QR or LU decompositions.
\end{remark}
\begin{remark}
	If the $Z_k$'s are independent hence $\lambda$ can be determined easily. Indeed, in that case the matrix $D(Z)$ is diagonal. Then, the $\lambda_{k}$'s are given by
	\begin{equation*}
		\lambda_k = \frac{\Cov \big( f_k(Z_k), f(Z) \big)}{\V \big( f_k(Z_k) \big)}.
	\end{equation*}
\end{remark}

Now, we can define $\widehat{I}^{\lambda, N}_{_{M}}$ the associated Monte Carlo estimator of $I^{\lambda, N}$
\begin{equation*}
	\widehat{I}^{\lambda, N}_{_{M}} = \frac{1}{M} \sum_{m=1}^M \Bigg( f(Z^m) - \sum_{k=1}^{d} \lambda_k f_k(Z_k^m) \Bigg) + \sum_{k=1}^{d} \lambda_k \E \big[ f_k(\widehat{Z}^N_k) \big].
\end{equation*}
One notices that $\E \big[ I - I^{\lambda, N} \big] \neq 0$, with bias equal to $\sum_{k=1}^{d} \lambda_k \big( \E \big[ f_k(\widehat{Z}^N_k) \big] - \E \big[ f_k(Z_k) \big] \big) $. However the quantity we are really interested by is not the bias but the \textit{MSE} (Mean Squared Error), yielding a \textit{bias-variance decomposition}
\begin{equation*}
	\textrm{MSE} (\widehat{I}^{\lambda, N}_{_{M}}) = \underbrace{ \Bigg( \sum_{k=1}^{d} \lambda_k \Big( \E \big[ f_k(\widehat{Z}^N_k) \big] - \E \big[ f_k(Z_k) \big] \Big) \Bigg)^{2}}_{\textit{bias}^2} + \frac{1}{M} \underbrace{ \V \Bigg( f(Z) - \displaystyle \sum_{k=1}^{d} \lambda_k f_k(Z_k) \Bigg) }_{\textit{Monte Carlo variance}}.
\end{equation*}
Our aim is to minimize the cost of the Monte Carlo simulation for a given \textit{MSE} or upper-bound of the \textit{MSE}. Consequently, for a given Monte Carlo estimator $\widehat{I}^{\lambda, N}_{_{M}}$ our minimization problem reads
\begin{equation}\label{EB:minimizationproblemMC}
	\inf_{MSE(\widehat{I}^{\lambda, N}_{_{M}}) \leq \epsilon^2} \textrm{Cost} (\widehat{I}^{\lambda, N}_{_{M}}).
\end{equation}
Let $\kappa = \textrm{Cost}(f(z))$ for a given $z \in \R^d$, the cost of a standard Monte Carlo estimator $\widehat{I}_{_{M}}$ of size $M$ is $\textrm{Cost} (\widehat{I}_{_{M}}) = \kappa M$. In our controlled case, if we neglect the cost for building an optimal quantizer, the global complexity associated to the Monte-Carlo estimator $\widehat{I}^{\lambda, N}_{_{M}}$ is given by
\begin{equation*}
	\textrm{Cost} ( \widehat{I}^{\lambda, N}_{_{M}} ) = \kappa \big( (d+1) M + d N \big)
\end{equation*}
where the cost of the computation of $f(z) - \lambda \sum_{k=1}^{d} f_k(z)$ is upper-bounded by $(d+1)\kappa$ whereas $\kappa d N$ is the cost of the quantized part. Indeed, there is $d$ expectations of functions of $N$-quantizers to compute, inducing a cost of order $\kappa d N$. Some optimizations can be implemented when computing $f_k(z)$, in that case $\textrm{Cost}(f_k(z)) < \kappa$. So, \eqref{EB:minimizationproblemMC} becomes
\begin{equation*}
	\inf_{MSE( \widehat{I}^{\lambda, N}_{_{M}} ) \leq \epsilon^2} \kappa \big( (d+1) M + d N \big).
\end{equation*}
Moreover, using the results in the first part of the paper concerning the weak error, we could define an upper-bound for the $MSE( \widehat{I}^{\lambda, N}_{_{M}})$, indeed if each $f_k$ is in a class of function where the weak error of order two is attained when using a quantization-based cubature formula then
\begin{equation*}
	MSE( \widehat{I}^{\lambda, N}_{_{M}} ) = \Bigg( \sum_{k=1}^{d} \lambda_k \Big( \E \big[ f_k(\widehat{Z}^N_k) \big] - \E \big[ f_k(Z_k) \big] \Big) \Bigg)^{2} + \frac{\sigma_{\lambda}^2}{M} \leq \frac{C}{N^4} + \frac{\sigma_{\lambda}^2}{M} \\
\end{equation*}
with $\sigma_{\lambda}^2 := \V \big( f(Z) - \sum_{k=1}^{d} \lambda_k f_k(Z_k) \big)$. Now, our minimization problem becomes
\begin{equation*}
	\inf_{ \frac{C}{N^4} + \frac{\sigma_{\lambda}^2}{M} \leq \epsilon^2} \kappa \big( (d+1) M + d N \big).
\end{equation*}
$\frac{C}{N^4}$ corresponds to the squared empirical bias and $\frac{\sigma_{\lambda}^2}{M}$ to the empirical variance, hence a standard approach when dealing with this kind of problem, is to equally divide $\epsilon^2$ between the bias and the variance: $\frac{C}{N^4} = \frac{\epsilon^2}{2}$ and $\frac{\sigma_{\lambda}^2}{M} = \frac{\epsilon^2}{2}$ yielding
\begin{equation*}
	N = O(\epsilon^{-\frac{1}{2} }) \quad \textrm{ and } \quad M = O(\epsilon^{-2} ),
\end{equation*}
hence the cost would be of order $O(\epsilon^{-2} )$. However, as the cost is additive and in the case where $\sigma_{\lambda}^2$ is close to $\V \big( f(Z) \big)$, meaning that the control variate does not really reduce the variance, we want to reduce the bias as much as we can. So another idea could be to choose both terms $M$ and $N$ of order $O(\epsilon^{-2} )$, because the impact on the cost of the Monte Carlo is at least of this order. Then, we search $\theta \in (0,1)$ defined by
\begin{equation*}
	\theta \epsilon^2 = \frac{C}{N^4} \quad \textrm{ and } (1 - \theta) \epsilon^2 = \frac{\sigma_{\lambda}^2}{M},
\end{equation*}
such that the impact on the cost of the Monte Carlo part and the quantization part are of same order: $O(\epsilon^{-2})$. In that case, $\theta$ is given by
\begin{equation*}
	\left\{ \begin{matrix}
		\theta \epsilon^2 = \frac{C}{N^4} \\
		\kappa d N  = O(\epsilon^{-2})
	\end{matrix} \right.
	\qquad \Longrightarrow \qquad \theta = O(\epsilon^{6}).
\end{equation*}
In practice, we do take not that high value for $N$. Indeed, the bias converges to $0$ as $N^{-4}$, so taking optimal quantizers of size $200$ or $500$ is enough for considering that the bias is negligible compared to the residual variance of the Monte Carlo estimator.

\begin{remark}
	Now, if we consider that we have no closed-form for $\E [ Z_k ], \,k=1,\ldots,d$, then we need to approximate them by $m_k$ (this would impact the total cost of the method, as one would need to use a numerical method for computing the $m_k$'s but this can be done once and for all before estimating $\widehat{I}^{\lambda, N}_{_{M}}$). These approximations yield different control variates: the functions $\widetilde f_k(z) := f ( m_1, \dots, m_{k-1}, z, m_{k+1},\dots, m_d ) $, inducing a different \textit{MSE}
	\begin{equation*}
		MSE( \widehat{I}^{\widetilde \lambda, N}_{_{M}} ) = \Bigg( \sum_{k=1}^{d} \widetilde \lambda_k \Big( \E \big[ \widetilde f_k(\widehat{Z}^N_k) \big] - \E \big[ \widetilde f_k(Z_k) \big] \Big) \Bigg)^{2} + \frac{\widetilde \sigma_{\widetilde \lambda}^2}{M}
	\end{equation*}
	with $\widetilde \sigma_{\widetilde \lambda}^2 := \V \big( f(Z) - \sum_{k=1}^{d} \widetilde \lambda_k \widetilde f_k(Z_k) \big)$ and $\widetilde \lambda_k, \,k=1,\ldots,d$. Finally, we can conclude in the same way as before if the $\widetilde f_k$'s are in a class of function where the weak error of order two is attained when using a quantization-based cubature formula.
\end{remark}

\subsection{Numerical results}

Let $(S_t)_{t\in [0,T]}$ be a geometric Brownian motion representing the dynamic of a \textit{Black-Scholes} asset between time $t=0$ and time $t=T$ defined by
\begin{equation*}
	S_t = S_0 \e^{ ( r - \sigma^2/2 ) t + \sigma W_t }
\end{equation*}
with $(W_t)_{t \in [0, T]}$ a standard Brownian motion defined on a probability space $(\Omega, \A, \Prob)$, $r$ the interest rate and $\sigma$ the volatility. When considering to use optimal quantization with a Black-Scholes asset, we have two possibilities: either we take an optimal quantizer of a normal distribution as $W_T \sim \N (0, T )$ or we build an optimal quantizer of a log-normal distribution as $\log (\e^{ ( r - \sigma^2/2 ) T + \sigma W_T } ) \sim \N \big( ( r - \sigma^2/2 ) T, \sigma^{2} T \big)$. In this part we consider both approaches since each one has its benefits and drawbacks.

Optimal Quantizers of log-normal random variables need to be computed each time we consider different parameters for the Black-Scholes asset. Indeed, the only operations preserving the optimality of the quantizers are translations and scaling. However, this transformations are not enough if one wishes to build an optimal quantizer of a Log-Normal random variables with parameters $\mu$ and $\sigma$ from an optimal quantizer of a standardized Log-Normal random variable. However, if one looses time by computing for each set of parameters an optimal quantizer for the log-normal random variable, it gains in precision.

Now, if we consider the case of optimal quantizers of normal random variables, we loose in precision because we do not quantize directly our asset but the optimal quantizers of normal random variables can be computed once and for all and stored on a file. Indeed, we can build every normal random variable from a standard normal random variable using translations and scaling. Moreover, high precision grids of the $\N(0,1)$-distribution are in free access for download at  the website: \url{www.quantize.maths-fi.com}.

Substantial details concerning the optimization problem and the numerical methods for building quadratic optimal quantizers can be found in \cite{pages2018numerical,pages2003optimal,pages2004optimal,mcwalter2018recursive}. In our case, we chose to build all the optimal quantizers with the Newton-Raphson algorithm (see \cite{pages2003optimal} for more details on the gradient and Hessian formulas for the $\N(0,1)$-distribution and \cite{mcwalter2018recursive} for other distributions) modified with the Levenberg-Marquardt procedure which improves the robustness of the method.

\subsubsection{Vanilla Call}
The payoff of a Call expiring at time $T$ is
\begin{equation*}
	( S_T - K )_+
\end{equation*}
with $K$ the strike and $T$ the maturity of the option. Its price, in the special case of \textit{Black-Scholes} model, is given by the following closed formula
\begin{equation}\label{EB:price_call}
	I_0 := \E \big[ \e^{-r T} ( S_T - K )_+ \big] = Call_{_{BS}}(S_0, K, r, \sigma, T) = S_0 \N (d_1) - K \e^{-r T} \N (d_2)
\end{equation}
where $\N(x)$ is the cumulative distribution function of the standard normal distribution, $d_1 := \frac{\log( S_0/K ) + (r + \sigma^2/2) T}{\sigma \sqrt{T}}$ and $d_2 := d_1 - \sigma \sqrt{T}$. Although the price of a Call in the Black-Scholes model can be expressed in a closed form, it is a good exercise to test new numerical methods against this benchmark. We compare the use of optimal quantizers of normal distribution, when one quantizes the law of the Brownian motion at time $T$ and log-normal distribution when one quantizes directly the law of the asset $S_T$ at time $T$.

In the first case, we can rewrite $I_0$ as a function of a random variable $Z$ with a $\N (0,1)$-distribution, namely a normal distributed random variable,
\begin{equation*}
	\E \big[ \e^{-r T} ( S_T - K )_+ \big] = \E \big[ f(Z) \big]
\end{equation*}
where $f(x) := \e^{-r T} ( s_0 \e^{(r-\sigma^2/2)T + \sigma \sqrt{T} x} - K )_+ $ is continuous with a piecewise-defined locally-Lipschitz derivative, with respect to the function $g(x) = \e^{ \sigma \sqrt{T} \vert x \vert}$.

In the second case, we have
\begin{equation*}
	\E \big[ \e^{-r T} ( S_T - K )_+ \big] = \E \big[ \varphi (S_T) \big]
\end{equation*}
where $\varphi (x) := \e^{-r T} ( x - K )_+ $ is piecewise affine with one break of affinity.

The Black-Scholes parameters considered are
\begin{equation*}
	s_0 = 100, \qquad r = 0.1, \qquad \sigma = 0.5,
\end{equation*}
whereas those of the Call option are $ T = 1 $ and $ K=80 $. The reference value is 34.15007. The first graphic in the Figure \ref{EB:fig:call} represents the weak error between the benchmark and the quantization-based approximations in function of the size of the grid: $ N \longmapsto \big\vert I_0 - \E \big[ f ( \widehat Z^N ) \big] \big\vert $ and $ N \longmapsto \big\vert I_0 - \E \big[ \varphi ( \widehat X^N ) \big] \big\vert $, the second represents the weak error multiplied by $ N^2 $ in function of $ N $: $ N \longmapsto N^2 \times \big\vert I_0 - \E \big[ f ( \widehat Z^N ) \big] \big\vert $ and $ N \longmapsto N^2 \times \big\vert I_0 - \E \big[ \varphi ( \widehat X^N ) \big] \big\vert $.
\begin{figure}[H]
	\centering
	\begin{subfigure}[t]{0.49\textwidth}
		\centering
		\includegraphics[width=1.\textwidth]{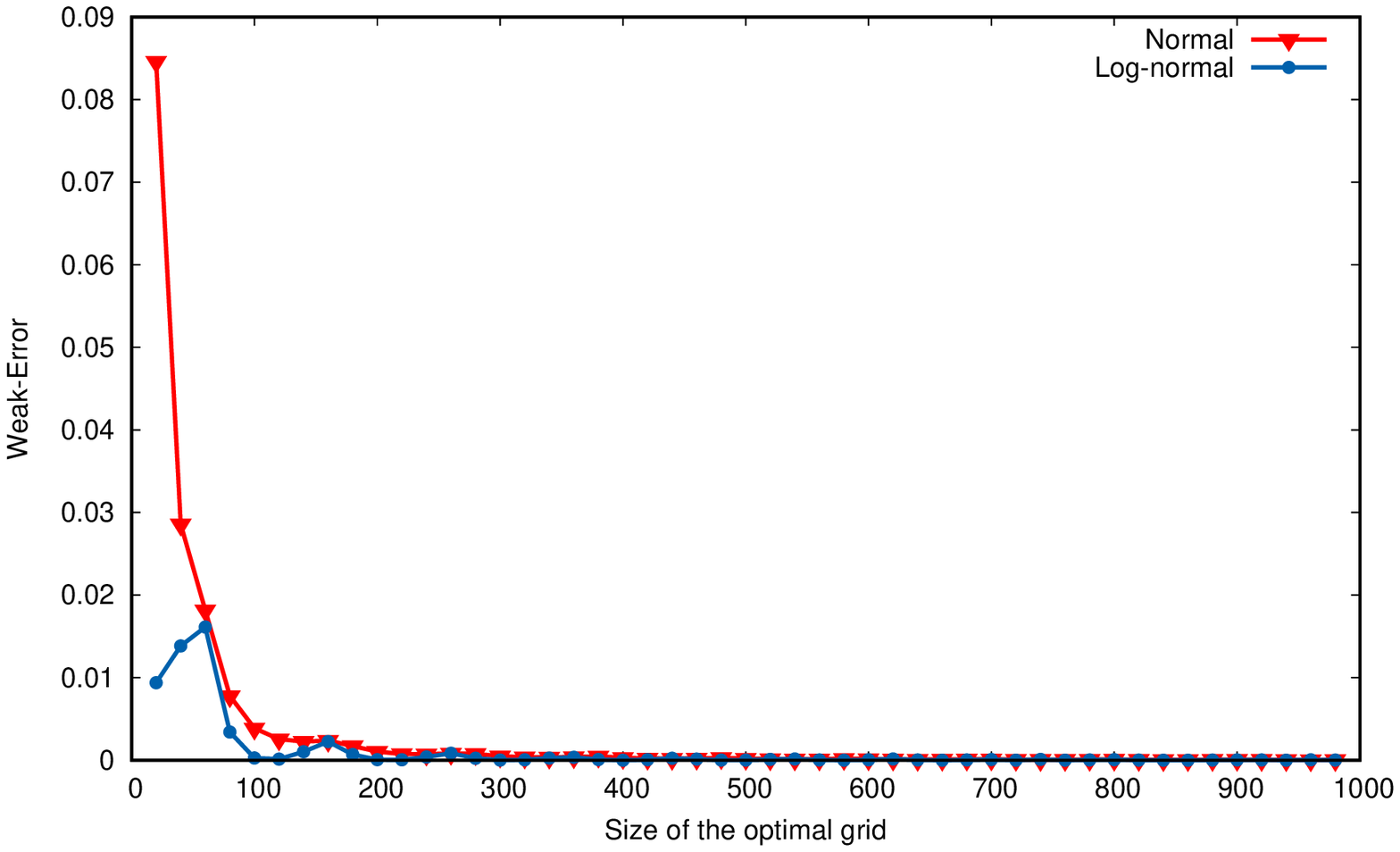}
		\caption{\textit{$N \longmapsto \big\vert I_0 - \E \big[ f ( \widehat Z^N ) \big] \big\vert$ (\textcolor{red}{\ding{116}}) and \\ $N \longmapsto \big\vert I_0 - \E \big[ \varphi ( \widehat X^N ) \big] \big\vert$ (\textcolor{NavyBlue}{\ding{108}}) }}
	\end{subfigure}%
	~
	\begin{subfigure}[t]{0.49\textwidth}
		\centering
		\includegraphics[width=1.\textwidth]{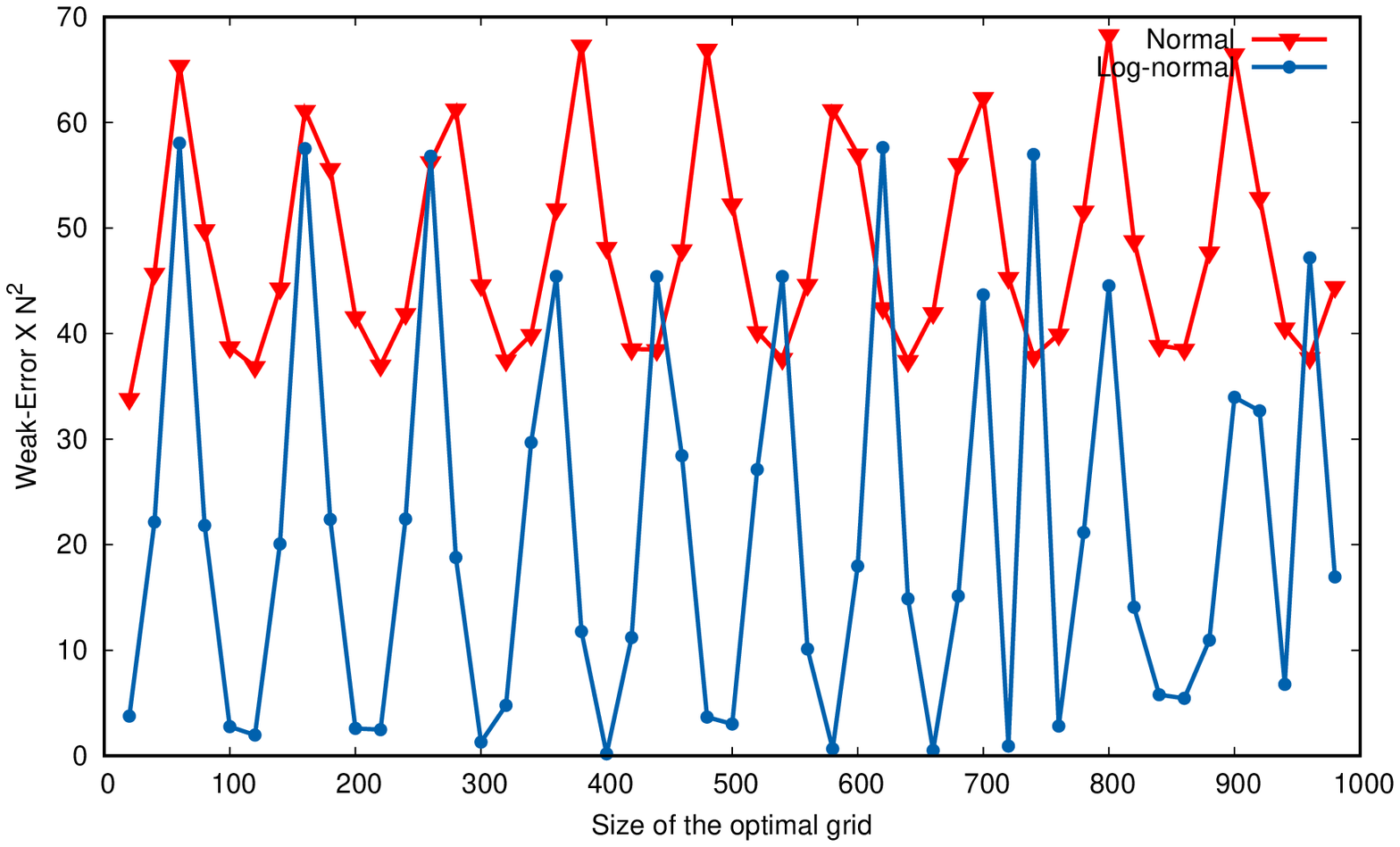}
		\caption{\textit{$N \longmapsto N^2 \times \big\vert I_0 - \E \big[ f ( \widehat Z^N ) \big] \big\vert$ (\textcolor{red}{\ding{116}}) and \\
				$N \longmapsto N^2 \times \big\vert I_0 - \E \big[ \varphi ( \widehat X^N ) \big] \big\vert$ (\textcolor{NavyBlue}{\ding{108}})}}
	\end{subfigure}
	\caption[Pricing of a Call option in a Black-Scholes model with optimal quantization.]{\textit{Call option in a Black-Scholes model.}}
	\label{EB:fig:call}
\end{figure}
First, we notice that both methods yield a weak-error of order $2$, as desired.
Second, if we look closely at the results the log-normal grids give a more precise price. However we need to build a specific grid each time we have a new set of parameters for the asset, whereas such is not the case when we choose to quantize the normal random variable, we can directly read precomputed grids with their associated weights in files.

\subsubsection{Compound Option}
The second product we consider is a Compound Option: a Put-on-Call. The payoff of a Put-on-Call expiring at time $T_1$ is the following
\begin{equation*}
	\Big( K_1 - \E \big[ \e^{-r(T_2-T_1)} (S_{T_2} - K_2)_+ \mid S_{T_1} \big] \Big)_+
\end{equation*}
with price
\begin{equation}\label{EB:priceputoncall}
	I_0 := \E \bigg[ \e^{-r T_1} \Big( K_1 - \E \big[ \e^{-r(T_2-T_1)} (S_{T_2} - K_2)_+ \mid S_{T_1} \big] \Big)_+ \bigg].
\end{equation}
The inner expectation can be computed, using the fact that $S_{T_2}$ is a \textit{Black-Scholes} asset and we know the conditional law of $S_{T_2}$ given $S_{T_1}$. Using \eqref{EB:price_call}, the value of the inner expectation is
\begin{equation*}
	\E \big[ \e^{-r(T_2-T_1)} (S_{T_2} - K_2)_+ \mid S_{T_1} \big] = Call_{_{BS}}(S_{T_1}, K_2, r, \sigma, T_2 - T_1).
\end{equation*}
Hence, the price of the Put-On-Call option in \eqref{EB:priceputoncall} can be rewritten as
\begin{equation*}
	I_0 = \E \Big[ \e^{-r T_1} \big( K_1 - Call_{_{BS}}(S_{T_1}, K_2, r, \sigma, T_2 - T_1) \big)_+ \Big].
\end{equation*}
The Black-Scholes parameters considered are
\begin{equation*}
	s_0 = 100, \qquad r = 0.03, \qquad \sigma = 0.2,
\end{equation*}
whereas those of the Put-On-Call option are $T_1 = 1/12$, $T_2 = 1/2$, $K_1=6.5$ and $K_2=100$. The reference value, obtained using an optimal quantizer of size $10000$ of the $ \N (0,1) $-distribution, is 1.3945704. As in the vanilla case, we compare the use of optimal quantizers of normal distribution and log-normal distribution. In the first case, we have
\begin{equation*}
	I_0 = \E \big[ f(Z) \big]
\end{equation*}
where $Z \sim \N (0,1)$ and $f(z) = \e^{-r T_1} \big( K_1 - Call_{_{BS}}(s_0 \e^{ (r -\sigma^2/2) T_1 + \sigma \sqrt{T_1} z}, K_2, r, \sigma, T_2 - T_1) \big)_+$, and in the second case
\begin{equation*}
	I_0 = \E \big[ \varphi (X) \big]
\end{equation*}
where $\log (X) \sim \N ( (r - \sigma^2/2) T, \sigma \sqrt{T} )$ and $\varphi(x) = \e^{-r T_1} \big( K_1 - Call_{_{BS}} (s_0 x, K_2, r, \sigma, T_2 - T_1) \big)_+ $.
The first graphic in Figure \ref{EB:fig:putoncall} represents the weak error between the benchmark and the quantization-based approximations in function of the size of the grid: $N \longmapsto \big\vert I_0 - \E \big[ f ( \widehat Z^N ) \big] \big\vert$ and $N \longmapsto \big\vert I_0 - \E \big[ \varphi ( \widehat X^N ) \big] \big\vert$, the second allows us to observe if the rate of convergence is indeed of order $2$.
\begin{figure}[H]
	\centering
	\begin{subfigure}[t]{0.49\textwidth}
		\centering
		\includegraphics[width=1.\textwidth]{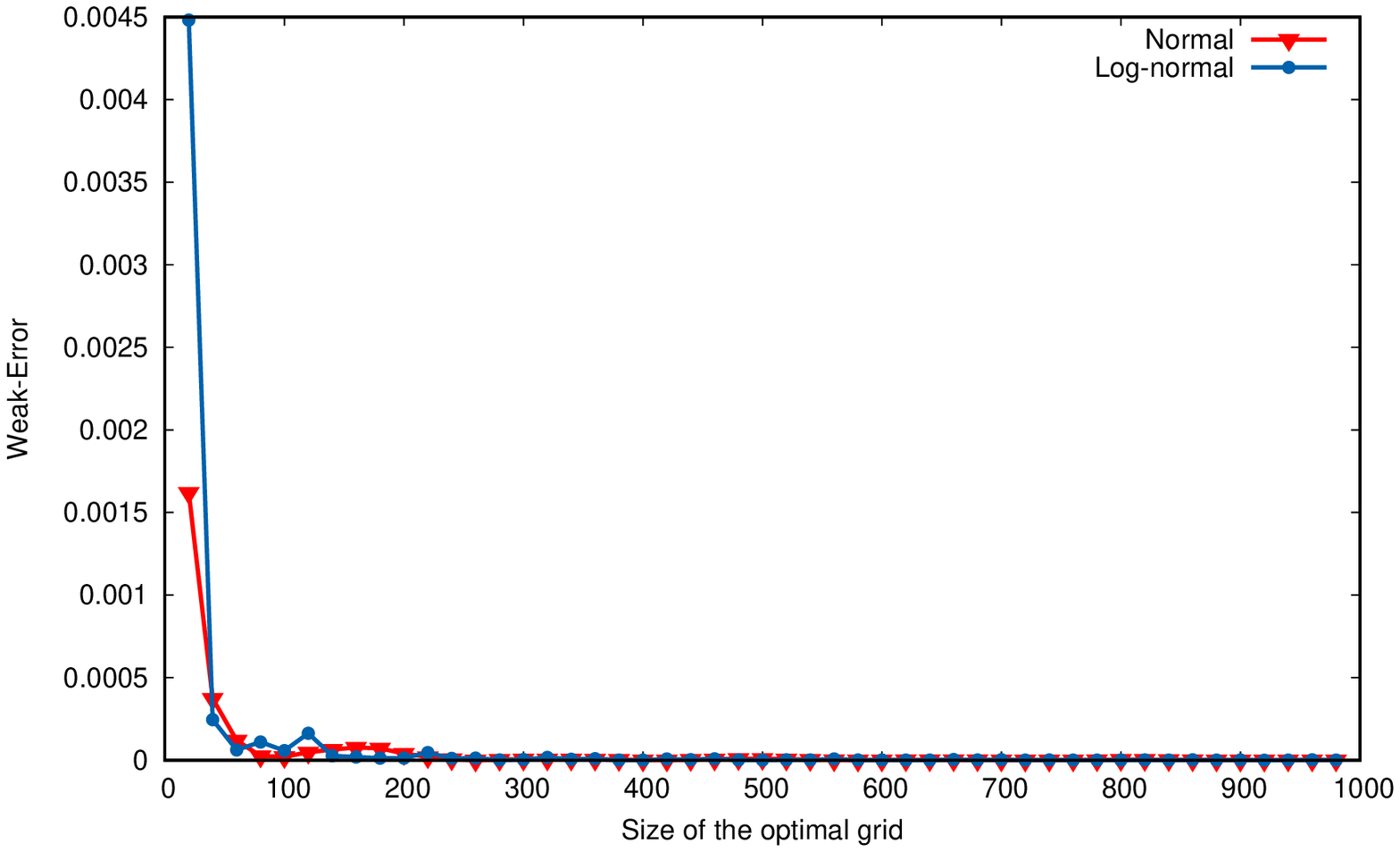}
		\caption{\textit{$N \longmapsto\big\vert I_0 - \E \big[ f ( \widehat Z^N ) \big] \big\vert$ (\textcolor{red}{\ding{116}}) and \\
				$N \longmapsto\big\vert I_0 - \E \big[ \varphi ( \widehat X^N ) \big] \big\vert$ (\textcolor{NavyBlue}{\ding{108}}) }}
	\end{subfigure}%
	~
	\begin{subfigure}[t]{0.49\textwidth}
		\centering
		\includegraphics[width=1.\textwidth]{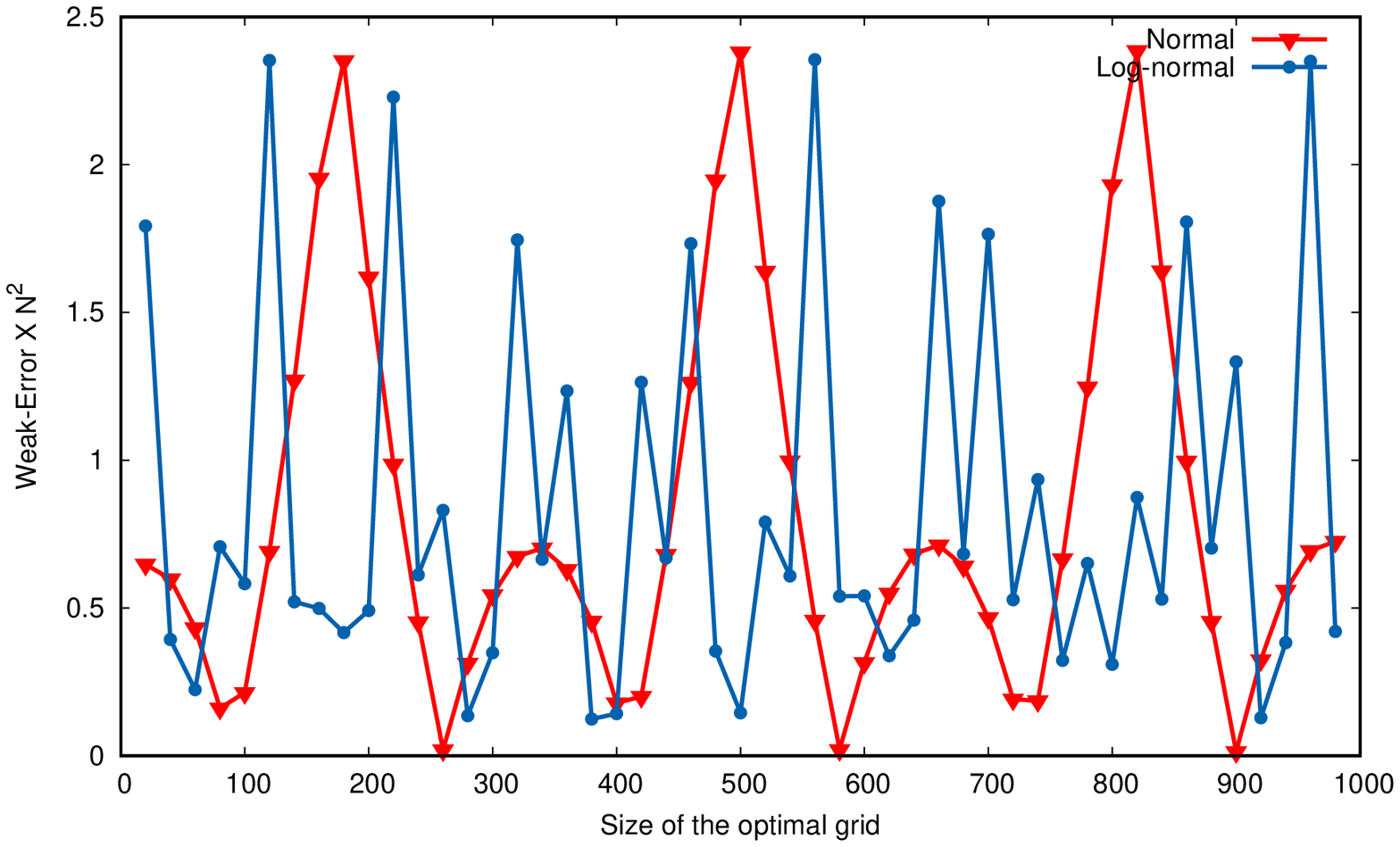}
		\caption{\textit{$N \longmapsto N^2 \times \big\vert I_0 - \E \big[ f ( \widehat Z^N ) \big] \big\vert$ (\textcolor{red}{\ding{116}}) and \\
				$N \longmapsto N^2 \times \big\vert I_0 - \E \big[ \varphi ( \widehat X^N ) \big] \big\vert$ (\textcolor{NavyBlue}{\ding{108}})}}
	\end{subfigure}
	\caption[Pricing of a Put-On-Call option in a Black-Scholes model with optimal quantization.]{\textit{ option in a Black-Scholes model.}}
	\label{EB:fig:putoncall}
\end{figure}
We notice that both methods yield a weak-error of order $2$ as desired, however it is not clear that one should use the log-normal representation of \eqref{EB:priceputoncall} in place of the Gaussian representation. Indeed, both constants in the rate of convergence are of the desired order and getting Gaussian optimal quantizers is much cheaper than building optimal quantizers of log-normal random variables. Hence, one should choose the Gausian representation as it is as precise as the log-normal one and is much cheaper.

\subsubsection{Exchange spread Option}

In this part, we consider a higher dimensional problem. Let two Black-Scholes assets $(S_T^i)_{i = 1,2}$ at time $T$ related to two Brownian motions $(W_T^i)_{i = 1, 2}$, with correlation $\rho \in [-1, 1]$. We are interested by an exchange spread option with strike $K$ with payoff
\begin{equation*}
	( S_T^1 - S_T^2 - K )_+
\end{equation*}
whose price is
\begin{equation}\label{EB:price_exchange_spread}
	I_0 := \E \big[ \e^{-rT} ( S_T^1 - S_T^2 - K )_+ \big].
\end{equation}
Decomposing the two Brownian motions into two independent parts, we have $(W_T^1, W_T^2) = \sqrt{T} ( \sqrt{1 - \rho^2} Z_1 + \rho Z_2, Z_2 )$, where $Z_1$ and $Z_2$ are two independent $\N (0,1)$-distributed Gaussian random variables. Now, pre-conditioning on $Z_2$ in \eqref{EB:price_exchange_spread} and using \eqref{EB:price_call}, we have
\begin{equation*}
	I_0 = \E \big[ \varphi (Z_2) \big]
\end{equation*}
where
\begin{equation*}
	\varphi(z) = Call_{_{BS}} ( s_0^1 \e^{- \rho^2 \sigma_1^2 T/2 + \sigma_1 \rho \sqrt{T} z}, s_0^2 \e^{ (r -\sigma^2_2/2) T + \sigma_2 \sqrt{T} z} + K, r, \sigma_1 \sqrt{1 - \rho^2}, T ).
\end{equation*}
The numerical specifications of the function $\varphi$ are as follows:
\begin{equation*}
	s_0^{i} = 100, \quad r = 0.02, \quad \sigma_{i} = 0.5, \quad \rho = 0.5, \quad T = 10, \quad K = 10.
\end{equation*}
In that case, the reference value is $53.552678$.

First, we look at the weak error induced by the quantization-based cubature formula when approximating \eqref{EB:price_exchange_spread}. We use optimal quantizers of the normal random variable $Z_2$. The quantization-based approximation is denoted $\widehat I_{N}$,
\begin{equation*}
	\widehat I_{N} := \E \big[ \varphi ( \widehat Z^N ) \big].
\end{equation*}
The first graphic in Figure \ref{EB:fig:exchange} represents the weak error between the benchmark and the quantization-based approximation in function of the size of the grid: $N \longmapsto \big\vert I_0 - \E \big[ \varphi ( \widehat Z^N ) \big] \big\vert$, the second plots $N \longmapsto N^2 \times \big\vert I_0 - \E \big[ \varphi ( \widehat Z^N ) \big] \big\vert$ and allows us to observe that the rate of convergence is indeed of order $2$.
\begin{figure}[H]
	\centering
	\begin{subfigure}[t]{0.49\textwidth}
		\centering
		\includegraphics[width=1.\textwidth]{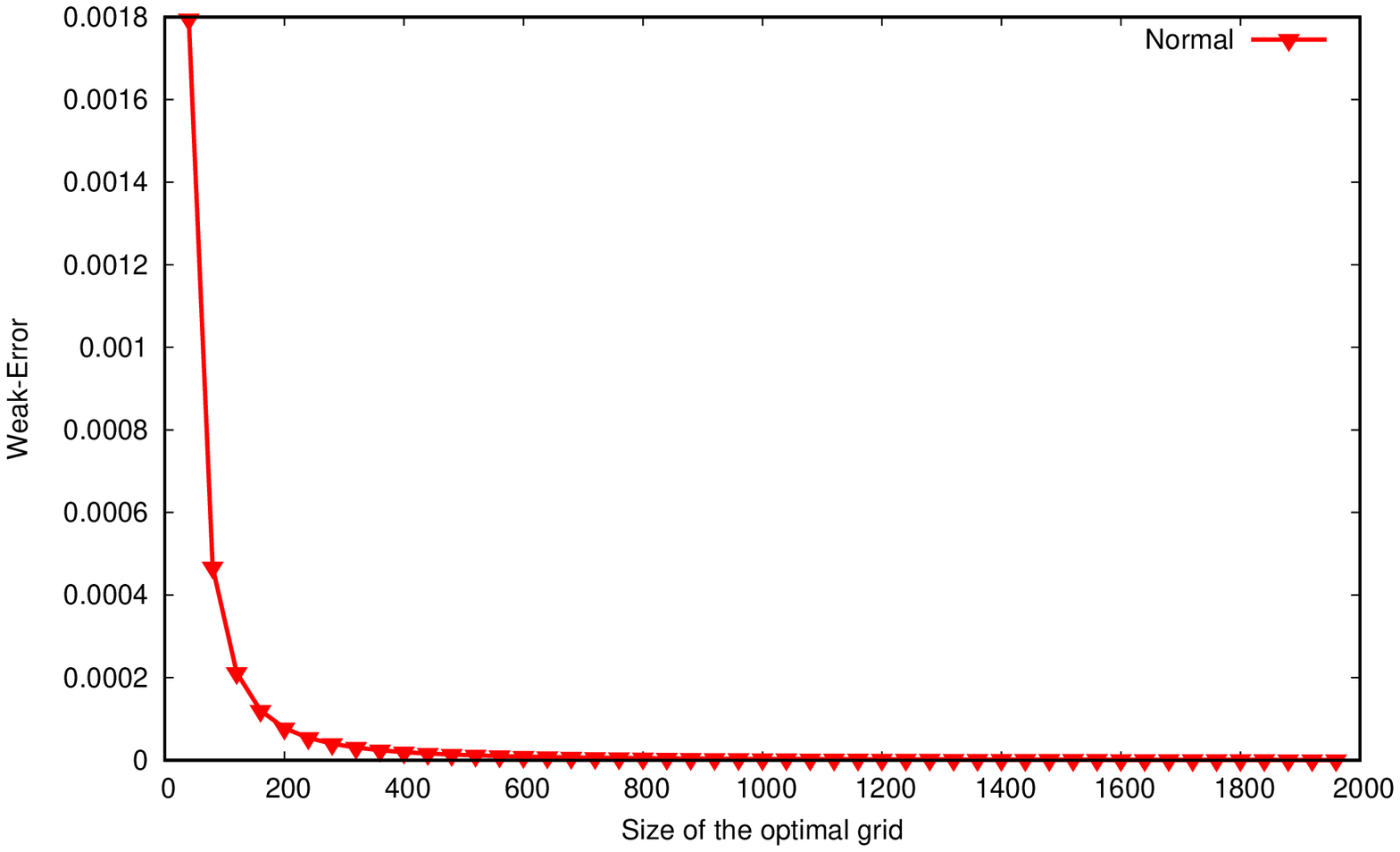}
		\caption{\textit{$N \longmapsto \big\vert I_0 - \E \big[ \varphi ( \widehat Z^N ) \big] \big\vert$ (\textcolor{red}{\ding{116}})}}
	\end{subfigure}%
	~
	\begin{subfigure}[t]{0.49\textwidth}
		\centering
		\includegraphics[width=1.\textwidth]{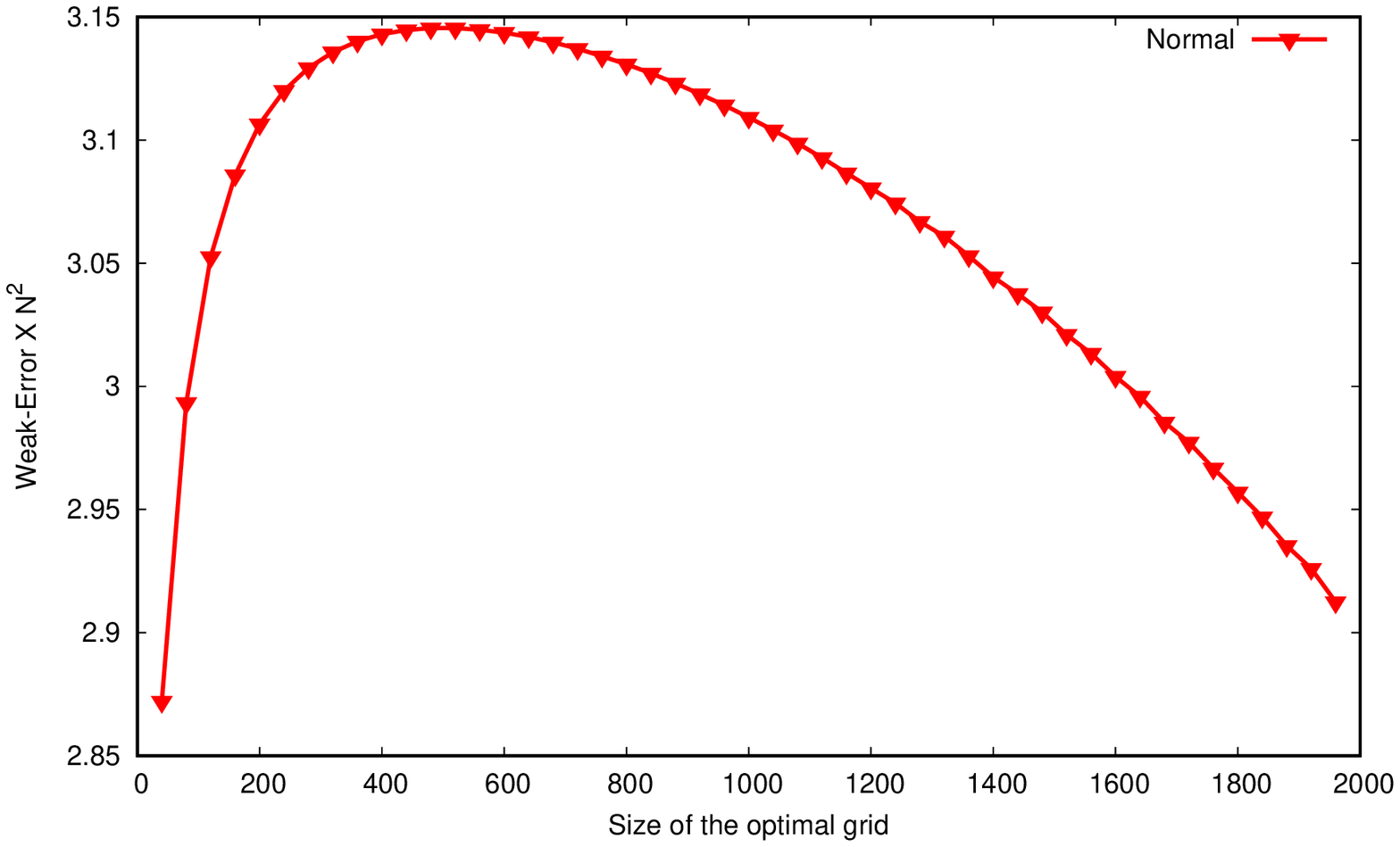}
		\caption{\textit{$N \longmapsto N^2 \times \big\vert I_0 - \E \big[ \varphi ( \widehat Z^N ) \big] \big\vert$ (\textcolor{red}{\ding{116}})}}
	\end{subfigure}
	\caption[Pricing of an Exchange spread option in a Black-Scholes model with optimal quantization.]{\textit{Exchange spread option pricing in a Black-Scholes model.}}
	\label{EB:fig:exchange}
\end{figure}

Now, noticing that $\varphi$ is a twice differentiable function with a bounded second derivative, we show that we can attain a weak error of order $3$ when using a Richardson-Romberg extrapolation denoted $\widehat I_{\widetilde N, N}^{RR}$ and defined in \eqref{EB:RR_quantif}.
\begin{figure}[H]
	\centering
	\begin{subfigure}[t]{0.49\textwidth}
		\centering
		\includegraphics[width=1.\textwidth]{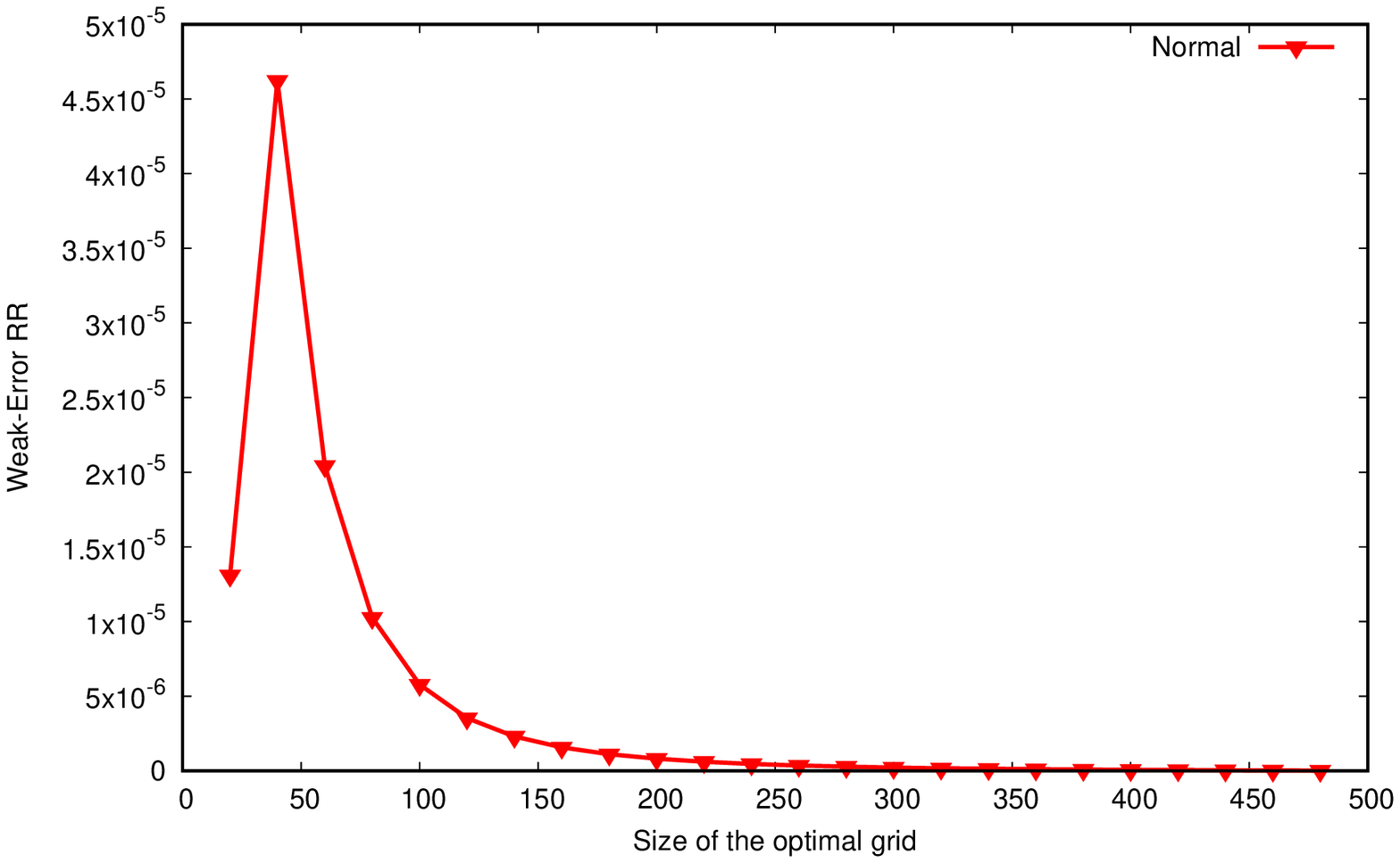}
		\caption{\textit{$N \longmapsto\vert I_0 - \widehat I_{\widetilde N, N}^{RR} \vert$ (\textcolor{red}{\ding{116}})}}
	\end{subfigure}%
	~
	\begin{subfigure}[t]{0.49\textwidth}
		\centering
		\includegraphics[width=1.\textwidth]{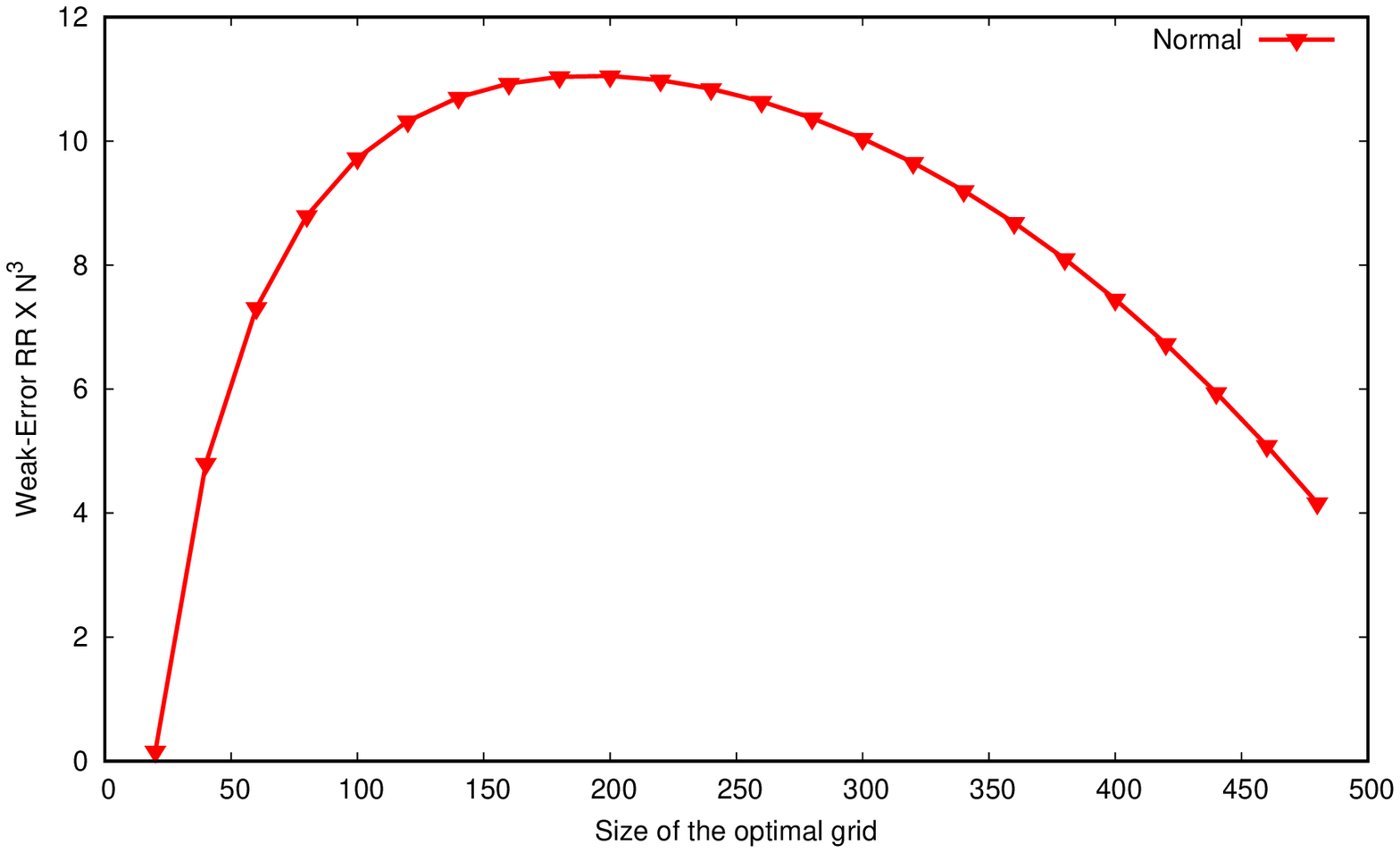}
		\caption{\textit{$N \longmapsto N^3 \times \vert I_0 - \widehat I_{\widetilde N, N}^{RR} \vert$ (\textcolor{red}{\ding{116}})}}
	\end{subfigure}
	\caption[Pricing of an Exchange spread option in a Black-Scholes model with optimal quantization (with Richardson-Romberg extrapolation).]{\textit{Richardson-Romberg extrapolation, with $\widetilde N = 1.2 \times N$, for Exchange spread option pricing in a Black-Scholes model.}}
	\label{EB:fig:exchangeRR}
\end{figure}

\subsubsection{Basket Option}

A typical financial product that allows to diversify the market risk and to invest in options is a basket option. The simplest one is an option on a weighted average of stocks. For example, if we consider an option on the FTSE index, this is a basket option where the assets are the companies defined in the description of the index and the weights are the market capitalization of each company at the time we built the index normalized by the sum on all market capitalizations.

In this part, we consider $d$ correlated assets $(S_T^k)_{k = 1, \dots, d}$ following a Black-Scholes model and the payoff we consider is
\begin{equation}\label{EB:payoff_basket}
	f(S_t^1, \dots, S_T^d) := \bigg( \sum_{k=1}^{d} \alpha_k S_T^k - K \bigg)_+
\end{equation}
whose price is
\begin{equation*}
	I_0 := \e^{-rT} \E \Bigg[ \bigg( \sum_{k=1}^{d} \alpha_k S_T^k - K \bigg)_+ \Bigg].
\end{equation*}
$I_0$ cannot be computed directly, hence we use a Monte Carlo estimator in order to approximate the expectation.
The standard estimator, denoted $\widehat{I}_M$, is the crude Monte Carlo estimator and is given by
\begin{equation*}
	\widehat{I}_M := \e^{-rT} \frac{1}{M} \sum_{m=1}^{M} \bigg( \sum_{k=1}^{d} \alpha_k S_T^{k,(m)} - K \bigg)_+
\end{equation*}
where $(S_T^{k,(m)})_{m = 1, \dots, M}$ are i.i.d. copies of $S_T^k$.
We compare the crude estimator to our novel approach based on a $d$-dimensional quantized control variates $\Xi^N$. In that case, $I_0$ is approximated by $I^N$ defined by
\begin{equation*}
	I^N := \e^{-rT} \E \Bigg[ \bigg( \sum_{k=1}^{d} \alpha_k S_T^k - K \bigg)_+ - \langle \lambda, \Xi^N \rangle \Bigg]
\end{equation*}
where $\Xi^N$ is defined later, yielding the following Monte Carlo estimator
\begin{equation*}
	\widehat I^{\lambda, N}_M := \e^{-rT} \frac{1}{M} \sum_{m=1}^{M} \bigg( \sum_{k=1}^{d} \alpha_k S_T^{k,(m)} - K \bigg)_+ - \langle \lambda, \Xi^{N,(m)} \rangle .
\end{equation*}
We propose two different control variates $\Xi^N$ based on optimal quantizers either of log-normal random variables or of Gaussian random variables.
\begin{enumerate}
	\item The control variate, denoted $\overline \Xi^{N}$, is defined by, $\forall k= 1, \dots, d$
	      \begin{equation*}
		      \overline \Xi^{N}_k := f(\E [ S_T^1 ], \dots, S_T^k, \dots, \E [ S_T^d ] ) - \E \big[ f(\E [ S_T^1 ], \dots, \widehat S_T^{k, N}, \dots, \E [ S_T^d ] ) \big]
	      \end{equation*}
	      where $(\widehat S_T^{k, N})_{k=1, \dots, d}$ are optimal quantizers of cardinality $N$ of $S_T^k$. In that case, the Monte Carlo estimator is denoted $\widehat{\overline{I}}^{\lambda, N}_M$.

	\item The control variate, denoted $\widetilde \Xi^{N}$, is using another representation of the payoff \eqref{EB:payoff_basket}, using $d$ Gaussian random variables i.i.d in place of the assets $S_T^k$ because the $d$ underlying correlated Brownian Motions can be expressed from $d$ rescaled independent Gaussian random variables, thus we define $\varphi$ our new representation for the payoff as
	      \begin{equation*}
		      \varphi (Z^1, \dots, Z^d) := f(S_T^1, \dots, S_T^d)
	      \end{equation*}
	      where $(Z^k)_{k=1, \dots, d}$ are i.i.d Gaussian random variables. Now, defining our control variates with the function $\varphi$, $\forall k= 1, \dots, d$
	      \begin{equation*}
		      \widetilde \Xi^{N}_k := \varphi(0, \dots, Z^k, \dots, 0) - \E \big[ \varphi(0, \dots, \widehat Z^{N}, \dots, 0) \big]
	      \end{equation*}
	      where $( \widehat Z^N )_{k=1, \dots, d}$ is an optimal quantizer of $Z \sim \N (0,1)$. In that case, the Monte Carlo estimator is denoted $\widehat{\widetilde{I}}^{\lambda, N}_M$.
\end{enumerate}

The Black-Scholes parameters considered are
\begin{equation*}
	s_0^{i} = 100, \qquad r = 2\%, \qquad \sigma_i = \frac{i}{d+1}, \quad \rho = 0.5,
\end{equation*}
and the specifications of the product are
\begin{equation*}
	K=100, \quad \alpha_i = \frac{2i}{d(d+1)}, \quad T=1
\end{equation*}
such that $\sum \alpha_i = 1$. The benchmarks used for the computation of the \textit{MSE} has been computed using a Monte Carlo estimator with control variate without quantization where the term $ \sum_{k=1}^{d} \E [ X_k ] $ is computed using Black-Scholes Call pricing closed formulas. The \textit{Mean Squared Error} of an estimator $I$ is computed using the formula
\begin{equation*}
	MSE(I) = \frac{1}{n} \sum_{i=1}^{n} ( I^{(i)} - I_0 )^{2}
\end{equation*}
where $(I^{(i)})_{i= 1, \dots, n}$ are $n$ independent copies of $I$.

Table \ref{EB:table:differents} compares three different types of Monte Carlo estimators: the standard (Crude) Monte Carlo estimator $\widehat{I}_M$, our novel Monte Carlo estimator with control variate based on optimal quantizers of Gaussian random variables $\widehat{\widetilde{I}}^{\lambda, N}_M$ and another one with optimal quantizers of log-normal random variables  $\widehat{\overline{I}}^{\lambda, N}_M$. The notation $n$ corresponds to the number of Monte Carlo used for computing the \textit{MSE}, $M$ is the size of each Monte Carlo and $N$ is the size of the optimal quantizers. The prices of reference for each $d$ are
\begin{itemize}
	\item for $d=2$: $14.2589$ $(\pm 0.0010)$,
	\item for $d=3$: $14.1618$ $(\pm 0.0015)$,
	\item for $d=5$: $13.9005$ $(\pm 0.0022)$,
	\item for $d=10$: $13.4979$ $(\pm 0.0034)$.
\end{itemize}

\begin{table}[!ht]
	\centering
	\begin{tabular}{ |c|l|c|c|c|c| }
		\hline
		\multicolumn{2}{|c|}{} & \multicolumn{2}{|c|}{$N=20$} & \multicolumn{2}{|c|}{$N=200$}                                                 \\
		\hline
		d                      & MC Estimator                 & Mean ($\pm 1.96 \times$std)   & MSE    & Mean ($\pm 1.96 \times$std) & MSE    \\
		\hline
		\multirow{3}{4em}{$d=2$}
		                       & Crude                        & $14.2695$ $(\pm 0.0662)$        & $0.1450$ & $14.2695$ $(\pm 0.0662)$      & $0.1450$ \\
		                       & CV Gaussian                  & $14.1017$ $(\pm 0.0399)$        & $0.0774$ & $14.2773$ $(\pm 0.0399)$      & $0.0530$ \\
		                       & CV Log-Normal                & $14.2351$ $(\pm 0.0078)$        & $0.0026$ & $14.2614$ $(\pm 0.0078)$      & $0.0020$ \\
		\hline
		\multirow{3}{4em}{$d=3$}
		                       & Crude MC                     & $14.1770$ $(\pm 0.0671)$        & $0.1492$ & $14.1770$ $(\pm 0.0671)$      & $0.1492$ \\
		                       & CV Gaussian                  & $14.0336$ $(\pm 0.0451)$        & $0.0837$ & $14.1685$ $(\pm 0.0451)$      & $0.0673$ \\
		                       & CV Log-Normal                & $14.1479$ $(\pm 0.0104)$        & $0.0038$ & $14.1674$ $(\pm 0.0104)$      & $0.0036$ \\
		\hline
		\multirow{3}{4em}{$d=5$}
		                       & Crude MC                     & $13.8803$ $(\pm 0.0720)$        & $0.1717$ & $13.8803$ $(\pm 0.0720)$      & $0.1717$ \\
		                       & CV Gaussian                  & $13.6686$ $(\pm 0.0562)$        & $0.1580$ & $13.8883$ $(\pm 0.0562)$      & $0.1044$ \\
		                       & CV Log-Normal                & $13.8797$ $(\pm 0.0151)$        & $0.0080$ & $13.9008$ $(\pm 0.0151)$      & $0.0076$ \\
		\hline
		\multirow{3}{4em}{$d=10$}
		                       & Crude MC                     & $13.5046$ $(\pm 0.0599)$        & $0.1186$ & $13.5046$ $(\pm 0.0599)$      & $0.1186$ \\
		                       & CV Gaussian                  & $13.2429$ $(\pm 0.0515)$        & $0.1527$ & $13.5113$ $(\pm 0.0515)$      & $0.0878$ \\
		                       & CV Log-Normal                & $13.4221$ $(\pm 0.0194)$        & $0.0181$ & $13.4983$ $(\pm 0.0194)$      & $0.0124$ \\
		\hline
	\end{tabular}
	\caption[Pricing of a Basket option in a Black-Scholes model with Monte Carlo using quantization-based control variates.]{\textit{$n=128$, $M=1e4$}}
	\label{EB:table:differents}
\end{table}

One remarks in Table \ref{EB:table:differents} the efficiency of the optimal quantization-based variance reduction method. The variance, in the best cases, can be divided by almost $100$ when using the optimal quantizers of Log-Normal random variables. Figure \ref{EB:fig:bias} shows the effect of $N$ (for $d=3$), the size the optimal quantizers, on the bias. The same seeds are used for all the Monte Carlo estimator, the only thing varying is $N$.

\begin{figure}[!ht]
	\centering
	\begin{subfigure}[t]{0.49\textwidth}
		\centering
		\includegraphics[width=1.\textwidth]{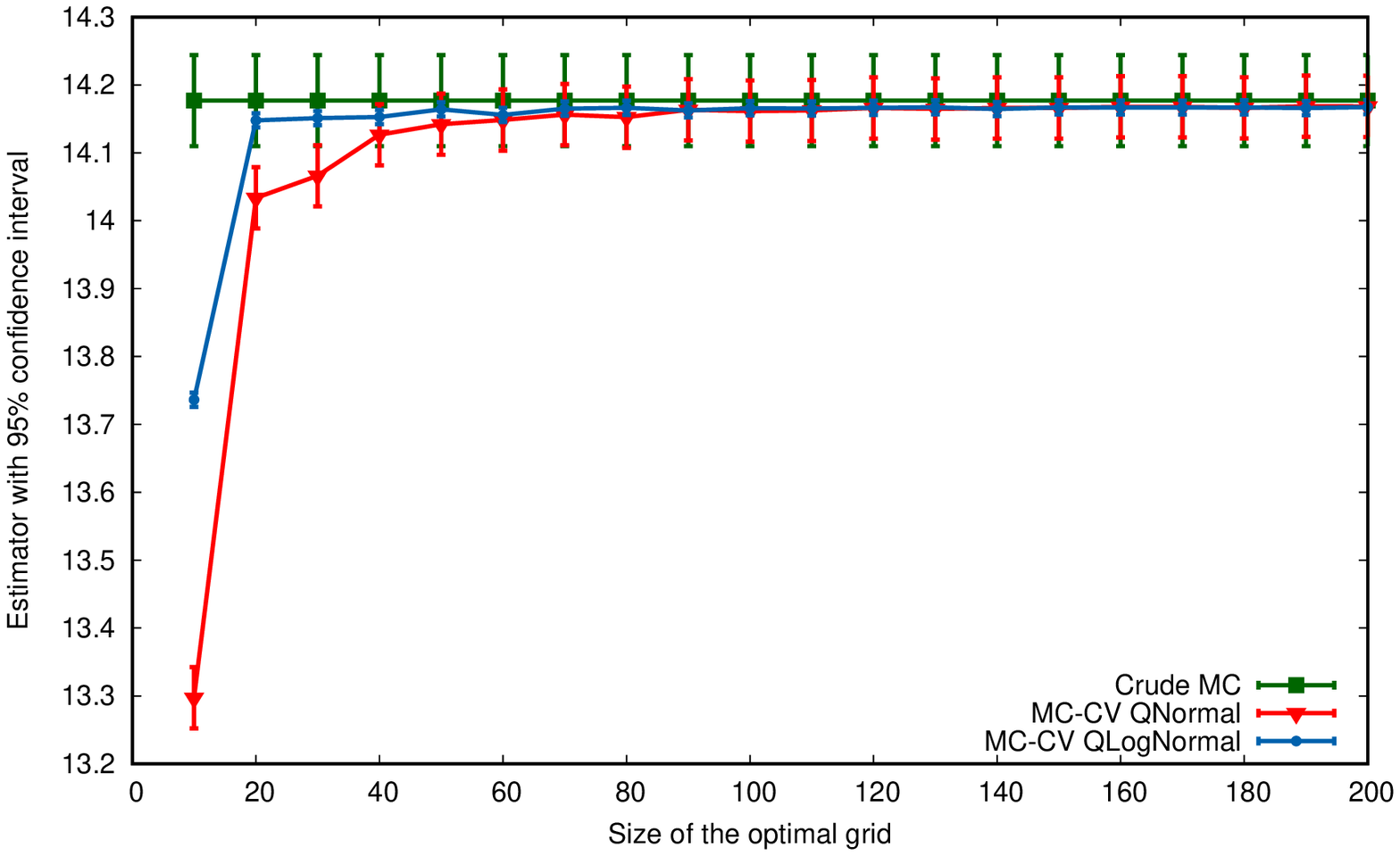}
		\caption{\textit{$N \longmapsto \vert I_0 - \widehat{\overline{I}}^{\lambda, N}_M \vert$ (\textcolor{red}{\ding{116}}), $N \longmapsto \vert I_0 - \widehat{\widetilde{I}}^{\lambda, N}_M \vert$ (\textcolor{NavyBlue}{\ding{108}}) and the Crude Monte Carlo estimator (\textcolor{ForestGreen}{\ding{110}}) with their associated confidence interval at $95\%$. } }
	\end{subfigure}%
	~
	\begin{subfigure}[t]{0.49\textwidth}
		\centering
		\includegraphics[width=1.\textwidth]{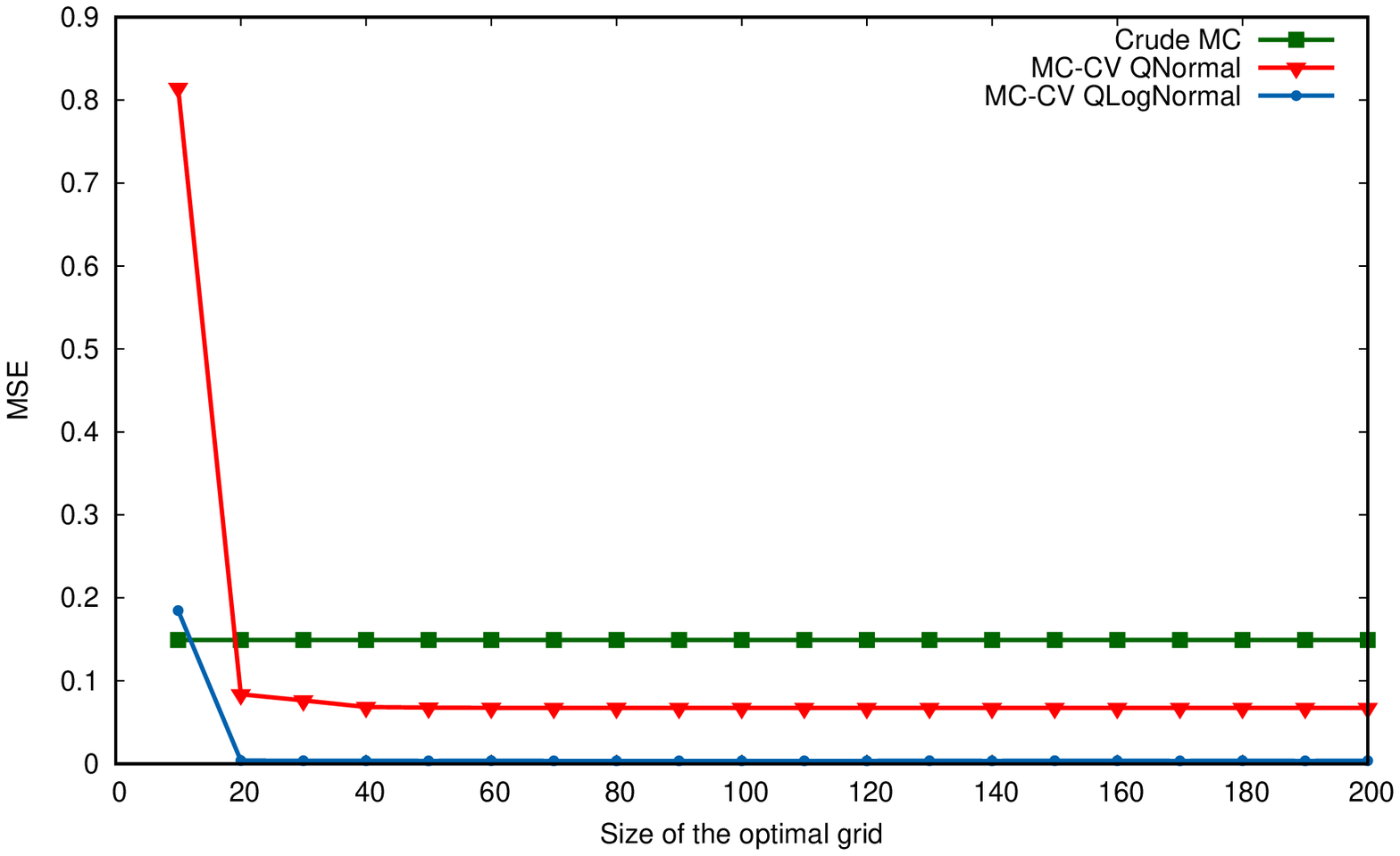}
		\caption{\textit{$N \longmapsto MSE (\widehat{I}_M)$ (\textcolor{ForestGreen}{\ding{110}}), $N \longmapsto MSE (\widehat{\overline{I}}^{\lambda, N}_M)$ (\textcolor{red}{\ding{116}}) and $N \longmapsto MSE (\widehat{\widetilde{I}}^{\lambda, N}_M)$ (\textcolor{NavyBlue}{\ding{108}}). } }
	\end{subfigure}
	\caption[Behavior in function of $N$ of the pricing of a Basket option in a Black-Scholes model with Monte Carlo using quantization-based control variates.]{\textit{$n=128$, $M=1e4$, $d=3$.}}
	\label{EB:fig:bias}
\end{figure}

\section*{Acknowledgment}
The authors wish to thank Pauline Corblet and Eric Tea for their useful feedback. The PhD thesis of Thibaut Montes is funded by a CIFRE grand from The Independent Calculation Agent (The ICA) and French ANRT.

\nocite{*}
\bibliography{bibli}
\bibliographystyle{alpha}

\end{document}